\def\tooee{LaTeX2e}
\ifx\fmtname\tooee
   \documentclass[12pt,leqno]{article}\usepackage{pb-diagram}
\else
   \documentstyle[12pt,pb-diagram]{article}
\fi
\usepackage[dvips]{graphicx}
\usepackage{amsmath,amssymb,amsthm,amscd}
\usepackage[mathscr]{eucal}
\newtheorem{thm}{Theorem}
\numberwithin{thm}{section}
\newtheorem{lemma}[thm]{Lemma}
\newtheorem{corollary}[thm]{Corollary}
\newtheorem{deff}[thm]{Definition}
\newtheorem{rmk}[thm]{Remark}
\newtheorem{prop}[thm]{Proposition}
\numberwithin{equation}{section}
\newcommand{\uG}{\bold G}
\newcommand{\uT}{\bold T}
\newcommand{\uB}{\bold B}
\newcommand{\uU}{\bold U}
\newcommand{\uA}{\bold A}
\newcommand{\uM}{\bold M}
\newcommand{\uN}{\bold N}
\newcommand{\uP}{\bold P}
\newcommand{\uH}{\bold H}
\newcommand{\sN}{\cal N}

\newcommand{\sC}{{\cal C}}
\newcommand{\sA}{{\cal A}}
\newcommand{\Res}{\operatornamewithlimits{Res}}
\newcommand{\diag}{\operatorname{diag}}
\begin{document}
\sloppy
\parskip=6pt
\title{The tempered spectrum of quasi-split classical groups III:\
 The odd orthogonal groups}
\author{D.~Goldberg\thanks{Partially supported by the NSF Grant DMS 9801340}\ \ \& F.~Shahidi\thanks{Partially supported by the NSF Grant DMS0200325 and a Guggenheim Fellowship}}
\maketitle
%\tableofcontents
%\listoffigures
{\bf Introduction}

\noindent
We continue our study of the tempered spectrum of quasi-split classical groups. Here we examine the
case of the special orthogonal groups of odd dimension.
While this is the last of the 
classical groups to be examined, it is the first for which our results address
the tempered spectrum whose supercuspidal support is an arbitrary maximal parabolic subgroup, as we describe below.
We continue to see the connection between 
poles of local Langlands $L$--functions, reducibility of parabolically induced from supercuspidal representations, and the theory of twisted endoscopy.
The recent progress in automorphic transfer and the local Langlands conjecture
allows us to get more precise results than in previous cases. In particular, we can show that poles of the local Rankin-product $L$--functions are determined by local components of automorphic transfer, and for the most interesting case of $GL_{2n}\times SO_{2n+1},$ the pole should be given precisely by this data.  That we can also resolve 
reducibility for $GL_k\times SO_{2n+1},$ for all $n$ and $k,$ stands in contrast to earlier cases, where we needed
some restrictions.

We let $M\simeq GL_n\times SO_{2m+1}$ be an arbitrary maximal Levi subgroup of
$G= SO_{2r+1},$  with $m+n=r$.
The main object of study for us is the 
standard intertwining operators and their poles. 
If $\tau'\otimes\tau$ is an irreducible
supercuspidal representation of $M$, then (when $\tau$ is generic)
the poles of the intertwining operators are those of the product of
two local $L$--functions, $L(2s,\tau',sym^2)$ and $L(s,\tau'\times\tau)$, the latter
being the Rankin product $L$--function.
The intertwining operators have simple poles, and hence at most one of the $L$--functions has a pole at $s=0$.
Furthermore, the operators
are entire if $\tau'\not\simeq\tilde\tau'$.
We will therefore restrict ourselves to the case $\tau'\simeq\tilde\tau'.$
The symmetric square $L$--functions are studied in \cite{Sh2}, and it is shown there that those irreducible self dual 
supercuspidal representations which do not have a pole of this $L$--function, give a pole of $L(s,\tau',\wedge^2),$
the exterior square $L$--function (cf \cite{Sh2}).  It is the subject
matter of \cite{He2} that these $L$--functions are the Artin $L$--functions 
$L(s,sym^2\phi')$ and $L(s,\wedge^2\phi'),$ where
$\phi'$ is the Langlands parameter attached to $\tau'$ 
as in \cite{HT,He}.   Thus, such irreducible supercuspidal representations must have the property that the
corresponding parameter given by \cite{HT,He} must factor through the symplectic group, as described in \cite{Sh2}.

The theory of $R$-groups \cite{Knapp-Stein,Silberger} reduces the
classification problem for the tempered spectrum to the case of maximal
parabolic subgroups, along with the combinatorial
problem of determining the $R$-groups themselves.  For
classical groups, the $R$-groups in question have
been determined \cite{G3,G4}.
The case of the Siegel Levi was addressed by the second named author
for $G=Sp_{2n}$ or $SO_n$ in \cite{Sh2}.
The first  named author applied the methods of \cite{Sh2}
to the Siegel Levi subgroups of quasi-split unitary groups \cite{G1,G2}.
The case of an arbitrary maximal parabolic subgroup was studied in \cite{Sh3}
 for (split) $SO_{2n}$, \cite{GS1}
for $Sp_{2n}$ and quasi-split $SO_{2n}^*$ and in \cite{GS2} for quasi--split unitary groups.
However, in each case, only the part of tempered spectrum with
supercuspidal support on certain maximal parabolic subgroups were resolved.
In particular,
only those parabolics with Levi components of the form $GL_{2n}\times SO_{2m}$,
$GL_{2n}\times Sp_{2m}$, $GL_{2n}\times SO_{2m}^*$, $\text{Res}_{E/F} GL_{2m}\times U_{m,m}$,
or $\text{Res}_{E/F} GL_{2n+1}\times U_{m,m+1}$. In short, the dimension of the general linear component must have
the same parity as the dimension of the classical group.  The reason for this restriction is our inability to
successfully analyze the norm maps in the case where the dimensions of the two groups have opposite parities.

In the case under current study, we are, in fact, able to give a description of
the tempered spectrum with supercuspidal support in any maximal 
parabolic subgroup.  We accomplish this by noting that the regular semisimple elements of $SO_{2m}$
and $SO_{2m+1}$ are in bijective correspondence, and this bijection is well behaved with respect to the norm map we
define here (cf. Lemma 2.5, Corollary 2.6, and Secton 3). 
We again use the theory of twisted endoscopy, as described
in
\cite{KS}, and to do so we need to understand the norm map from twisted conjugacy classes of $GL_n$ to conjugacy
classes of
$SO_{2m+1}$. However, the argument of Lemma 5.9 of \cite{Sh3}, which could also be applied to Lemmas 3.11 of
\cite{GS1} and 2.12 of \cite{GS2} is not directly applicable here. Hence a different approach is needed. We use the
fact that the tori of $SO_{2m+1}$ and
$SO_{2m}$ are in bijection  in a very straightforward way (cf Corollary 2.6 and Proposition 2.7).
This allows us to construct an explicit proof that the image of the norm map includes all semisimple classes when $n=2m+1$.
We are then able to further exploit the
matching of tori, and its compatibility with the norm map to show that the same result holds when $n=2m$ (cf. Section 3).
This is then enough to apply the methods of previous cases to arbitrary parabolic subgroups.
We remark that the case of greatest interest here, is that of $n=2m$, owing
to the local functoriality conjecture of Langlands \cite{C-K-PS-S,JS}.

Our understanding of the norm map allows us to interpret the poles of intertwining operators as sums of integrals in which the integrand is a product of a twisted orbital integral on $GL_n$ times an orbital integral on
$SO_{2m+1}$.
We further decompose this and regroup the sum to
obtain, as in previous cases, two parts.
The first of which we refer to as the main term, $R_{\bold G}$, 
and it comes from the Weyl integration formula.
The other part is given by singular terms.
These singular terms are of the form
$\displaystyle{|W(\bold T)|^{-1}\Res_{s=0}
\int_{T\setminus\Omega}\varphi_{\sA}(s,\gamma)|D(\gamma)|\,d\gamma}$,
with $\bold T$ a Cartan subgroup of $SO_{2\ell+1}$, for some
$\ell$, $\varphi_{\sA}$ an entire function of
$s$, and $\Omega$  an open compact subset of $T$ (cf Theorem 4.9 and Corollary 4.10).
It is quite remarkable that the residue in each of the singular terms above is independent of the choice of $\Omega$, and therefore depends only on the singular set.  This was noted as well for earlier cases (see \cite{GS2}).
At this time, we have no way of resolving these terms further.
Understanding these terms in some more
explicit way must be the focus of further study. We find
this extremely intriguing, and believe these terms must contain 
deep arithmetic information which, as yet, we have no techniques for extracting.

There is one other aspect of our study here which differs from our earlier reports.  
Due to the significant progress in the Langlands functoriality conjecture, we are able to give a better description of the relation between the poles of $L$--functions and the theory of twisted endoscopy.
In particular, we use the recent results of \cite{C-K-PS-S} and \cite{JS} to define a notion of  local automorphic transfer (cf Definition 5.2).
Using the properties of the automorphic transfer, we note that
$L(s,\tau'\times\tau)$ is entire when $n>2m.$
Furthermore, if $n=2m,$ $L(s,\tau'\times\tau)$ has a pole if and only if $\tau'$ is the local automorphic transfer of $\tau$ (cf 
Proposition 5.3 and \cite{JS}).
On the other hand, the non-vanishing of the
main term $R_{\bold G}$ is indicative of $\tau'$ coming from $\tau$ via twisted endoscopy (cf.~Definition 5.1).
This is our strongest evidence to date that
automorphic transfer is given locally by twisted endoscopy.

In Section 1 we give basic definitions, and give an initial characterization of the intertwining operator.
In Section 2, we examine the norm correspondence for the case $n=2m+1$.
In Section 3 we turn to the case $n=2m$, and then
apply those results, as well as those of Section 2 to describe the norm correspondence when $n>2m+1$.
In Section 4 we describe the poles of the intertwining operator for all $n$ and $m$, and prove our main theorem.
In Section 5 we examine the relation of the results of Section 4 to automorphic transfer and twisted endoscopy.  We remark that several of the proofs of the results in sections 1 and 2 are almost verbatim those of the corresponding results in \cite{Sh3, GS1, GS2}, except for minor details, such as carrying a sign all the way through.  For this reason we choose to omit some of these longer proofs, referring to the earlier results.

The authors would like to thank the Centre 
International de Recontres Math\'ematiques, in Luminy France, and the Park City Mathematics Institute, where much of this manuscript was written.
Both institutes provided pleasant environs in which to work and 
a high level of interesting and motivating mathematical activity.

\section{Preliminaries}
Let $F$ be a local nonarchimedean field of characteristic zero.
Fix the form
$w_n=\begin{pmatrix} 
&&&&&& 1\\ &&&&& .\\&&&&.\\&&&.\\ &&1\\ &1\\ 1
\end{pmatrix}\in GL_n(F)$, for any $n\geq 1$.
Note that $w_n={}^t w_n=w_n^{-1}$.
We let $\uG=\uG(r)=SO_{2r+1}$, defined with respect to the form $w_{2r+1}$.
The maximal split torus of diagonal elements is denoted by $\uT$,
$$
\uT=\left\{\begin{pmatrix}
x_1\\ &x_2\\ &&\ddots\\ &&&x_r\\ &&&&1\\ &&&&&x_r^{-1}\\ &&&&&&\ddots\\ &&&&&&& x_2^{-1}\\
&&&&&&&&x_1^{-1}\end{pmatrix}\Bigg|x_i\in\uG_m=GL_1\right\}
$$
We set $\uB=\uT\uU$ to be the upper triangular matrices in $\uG$.
The root system, $\Phi(\uG,\uT)$ is of type $B_r$, with simple roots $\{e_1-e_2,e_2-e_3,\ldots,e_{r-1}-e_r,e_r\}$.
We take the subset $\theta=\Delta-\{e_n-e_{n+1}\}$.
Setting $m=r-n$, we have
$$
\uA=\uA_\theta=\left\{\begin{pmatrix} xI_n\\ &I_{2m+1}\\ &&x^{-1}I_n \end{pmatrix}\bigg|x\in GL_1\right\},
$$
and $\uM=\uM_\theta=Z_{\uG}(\uA_\theta)=\left\{\begin{pmatrix} g\\ &h\\ &&\varepsilon(g)\end{pmatrix}\bigg|
\substack{h\in\uG(m)\\ g\in GL_n}\right\}$, where $\varepsilon(g)=w_n {}^tg^{-1} w_n^{-1}$.
Thus, $\uM\simeq GL_n\times SO_{2m+1}$.
Take $\uP=\uM\uN$, to be the standard (with respect to $(\uB,\uT)$), parabolic with Levi component $\uM$.
Then
$$
N=\Bigg\{\begin{pmatrix} I&X&Y\\ 0&I&X'\\ 0&0&I\end{pmatrix} \bigg|
Y+\tilde\varepsilon(Y)=XX' \Bigg\},
$$
where $X'=-w_{2m+1}{}^tX w_n$, and $\tilde\varepsilon(Y)=w_n{}^t Y w_n$.

We denote
\begin{equation}
Y+\tilde\varepsilon(Y)=XX'.
\end{equation}
Note that if $(X,Y)$ is a solution to (1.1), then
\begin{eqnarray*}
gY\varepsilon(g)^{-1}&+&\tilde\varepsilon(g Y\varepsilon(g)^{-1})=g
(Y+\tilde\varepsilon(Y))\varepsilon(g)^{-1}\\
&=&(gX)(X'\varepsilon(g)^{-1})=(gX)(gX)'.
\end{eqnarray*}
Thus if there is a solution to (1.1) for $Y$, then there is a solution for every element of the orbit of $Y$ under $\varepsilon$--twisted
adjoint action of $GL_n$.

Also, if $Y\in GL_n(F)$, and $(X,Y)$ satisfies (1.1), then $Y^{-1}+\tilde\varepsilon(Y^{-1})=Y^{-1}(Y+\tilde\varepsilon(Y))\tilde\varepsilon(Y^{-1})=Y^{-1} XX'\tilde\varepsilon(Y)^{-1}=(Y^{-1} X)(Y^{-1} X)'$.
So the $\varepsilon$--conjugacy classes for which (1.1) has a solution is closed under inversion.
We let $\sN$ be this collection of $\varepsilon$--conjugacy classes.

We fix $w_0=\begin{pmatrix} &&I_n\\ &(-1)^n I_{2m+1}\\ I_n\end{pmatrix}$.
Then $w_0$ represents the unique non--trivial element of the Weyl group $W(\uG,\uA)$.
Let $\overline{\uN}$ be the unipotent radical opposite to $\uN$.

\begin{lemma}Let $u=\begin{pmatrix} I&X&Y\\ 0&I&X'\\ 0&0&I\end{pmatrix}\in\uN$.
Then $w_0^{-1}u\in P\overline{N}$ if and only if $Y\in GL_n$, in which case
\begin{equation}
w_0^{-1}u=\begin{pmatrix} \varepsilon(Y)&-Y^{-1}X & I_n\\ 0&(-1)^n (I-X' Y^{-1} X)&(-1)^n X'\\ 0&0&Y\end{pmatrix}\begin{pmatrix} I_n&0&0\\ (Y^{-1}X)'&I&0\\ Y^{-1}&Y^{-1}X&I\end{pmatrix}
\end{equation}
\end{lemma}

\begin{proof}This is a straightforward matrix calculation.
\end{proof}

\begin{corollary}If $(X,Y)$ is a rational solution to (1.1), with $Y\in Gl_n(F)$, then $(-1)^n(I-X' Y^{-1}X)\in G(m)$.
\end{corollary}

We fix $\tau'\in\, ^o{\cal E}(GL_n(F))$, and $\tau\in\,^o{\cal E}(G(m)),$ where 
$^o{\cal E}(H)$ denotes the equivalence classes of irreducible admissible unitary supercuspidal representations of a
reductive $p$--adic group $H.$ We set $V(s,\tau'\otimes\tau)$ to be the space of the induced representation,
$\text{Ind}_P^G((\tau'\otimes|\det(\ )|_F^s)\otimes\tau)=I(s,\tau'\otimes\tau)$.
We wish to determine the reducibility of $I(s,\tau'\otimes\tau)$, and thus, we may assume that
$(\tau'\otimes\tau)^{w_0}\simeq\tau'\otimes\tau.$  This is equivalent to
$\tilde\tau'\simeq\tau',$ with $\tilde\tau'$ the smooth contragredient of $\tau'.$ Let $V(s,\tau'\otimes\tau)_0$ be
those elements of
$V(s,\tau'\otimes\tau)$ which are compactly supported in $\overline N\mod P$. We fix $f\in V(s,\tau'\otimes\tau)_0$,
with
$$
f\Bigg(\begin{pmatrix} I&0&0\\ X'\varepsilon(Y)&I&0\\ Y^{-1}&Y^{-1}X&I\end{pmatrix}\Bigg)=\xi_L(Y^{-1})\xi_{L'}(Y^{-1}X)\cdot
v'\otimes v,
$$
for some $v'\in V_{\tau'},\ v\in V_\tau,\ L,L'$ compact subsets of $M_n(F),\ M_{n\times 2m+1}(F)$, respectively, and $\xi_L,\ \xi_{L'}$ are characteristic functions.  Note that such functions span $V(s,\tau'\otimes\tau)_0$ over $G.$

Fix $\tilde v',\ \tilde v$ in $\tilde V_{\tau'}$ and $\tilde V_\tau$, respectively.
Let
\begin{eqnarray*}
\psi_{\tau'}(g)&=&\langle \tilde v',\ \tau'(g) v'\rangle,\ g\in GL_n(F),\text{ and}\\
f_\tau(h)&=&\langle\tilde v,\ \tau(g) v\rangle,\ \text{for }h\in G(m).
\end{eqnarray*}
We are interested in examining the standard intertwining operator
$$
A(s,\tau'\otimes\tau,w_0) f(g)=\int\limits_N f(w_0^{-1}ng)dn.
$$
We need only examine poles of $s\mapsto A(s,\tau'\otimes\tau,w_0)f(e),$ (Lemma 4.1 of \cite{Sh2}) for $f\in V(s,\tau'\otimes\tau)_0$.
Note that, for our choice of $f$ as above
\begin{eqnarray}
&&\langle \tilde v'\otimes \tilde v,\ A(s,\tau'\otimes\tau,\ w_0)f(e)\rangle=\\
&&\int\limits_{(X,Y)}\psi_{\tau'}(\varepsilon(Y)) f_\tau((-1)^n(I-X'Y^{-1}X))
|\det Y|_F^{-s-\langle\rho,\tilde\alpha\rangle}\xi(X,Y)d(X,Y)\nonumber
\end{eqnarray}
with $\xi(X,Y)=\xi_L(Y^{-1})\xi_{L'}(Y^{-1}X)$.
Here $(X,Y)$ is taken over all solutions to (1.1) with $Y\in GL_n(F)$.

\begin{lemma}Suppose $(X,Y)$ is a rational solution to (1.1).
If $(Xg)(Xg)'=XX'$, for some $g\in GL_{2m+1}(F)$, then $Xg=Xh$, for some $h\in SO_{2m+1}(F)$.
\end{lemma}

\begin{proof} The proof is essentially the same as Lemma 4.1 of \cite{Sh3}, Lemma 3.1 of \cite{GS1}, and Lemma 2.1 of \cite{GS2}.
\end{proof}

We have already seen that if $g\in GL_n(F)$, and $(X,Y)$ satisfies (1.1), then so does $(gX,\ gY\varepsilon(g)^{-1})$.
Hence the orbits $\{X\}\in GL_n(F)\backslash M_{n\times 2m+1}(F)/SO_{2m+1}(F)$.
Parameterize the $\varepsilon$--conjugacy classes for which (1.1) has a solution.
We say, in the case $(X,Y)$ satisfies (1.1) that $\{X\}$ parameterizes $\{Y^{-1}\}$.
Replacing $X$ with $gX$ leaves $(-1)^n (I-X' Y^{-1} X)$ unchanged if we replace $Y$ with $g Y\varepsilon (g)^{-1}$.

If $X_1=Xh$, with $h\in GL_{2m+1}(F)$, then by Lemma 1.3 we may assume $h\in SO_{2m+1}(F)$.
Then $X_1 X'_1=(Xh)(Xh)'=XX'$, so $(Xh,Y)$ is also a solution to (1.1), and 
\begin{eqnarray*}
&&(-1)^n (I-X'_1 Y^{-1} X_1)=(-1)^n (I-(Xh)' Y^{-1} (Xh)))=\\
&&(-1)^n (I-(-w_{2n+1}\, ^th ^t\!X w_n Y^{-1} Xh))=\\
&&(-1)^n (I- h^{-1} (-w_{2m+1}\,^t\!X w_n Y^{-1} X)h)=h^{-1}
((-1)^n I-X' Y^{-1} X)h
\end{eqnarray*}
so the conjugacy class of $(-1)^n (I-X' Y^{-1} X)$ is unchanged.

The following is now clear.

\begin{lemma}Suppose $\{X_1\}$, with $X_1=g Xh$ parameterizes $\{Y_1^{-1}\}$, with $Y_1=g Y\varepsilon (g)^{-1}$, (and $h\in GL_{2m+1}(F)$).
Then $X_1\in GL_n(F)X SO_{2n+1}(F)$.
\end{lemma}

\begin{lemma}Suppose $(X,Y)$ satisfies (1.1) with $Y$ invertible.
Then
\begin{description}
\item[(a)]$(I-X' Y^{-1} X) X'=-X' Y^{-1}\varepsilon (Y^{-1})$;
\item[(b)]$X(I-X' Y^{-1} X)=-\varepsilon (Y^{-1}) Y^{-1} X$.
\end{description}
\end{lemma}

\begin{proof}
The proof is in essence that of  Corollary 3.2 and Proposition 5.4(b) of \cite{Sh3}, Lemma 3.3 of \cite{GS1} or Lemma 2.3 of \cite{GS2}
\end{proof}

\begin{lemma}Suppose $X\in M_{n\times 2m}(F)$, and $U=F^n X$.
Let $H_X=\{h\in SO_{2m+1}(F)|Xh=g_h X$ for some $g_h\in GL_n(F)\}$.
If $\{0\}\subsetneq U\subsetneq F^{2m+1}$, then $H_X\subsetneq SO_{2m+1}(F)$.
\end{lemma}

\begin{proof}Suppose $h\in H_X$.
Let $u\in U$, and choose $v\in F^n$, with $u=vX$.
Then $uh=vXh=vg_h X\in U$.
Thus, $U$ is $H_X$--invariant.
If $H_X=SO_{2m+1}(F)$, then we know $U=\{0\}$, or $U=F^{2m+1}$.
\end{proof}

\begin{lemma}If $\{X\}\in GL_n(F)\backslash M_{n\times 2m+1}(F)/SO_{2m+1}(F)$, and $(X,Y)$ satisfies (1.1), then $X(I-X' Y^{-1} X)=-\varepsilon(Y^{-1})Y^{-1}X$, and $(-1)^n (I-X' Y^{-1} X)\in H_X$.
\end{lemma}

\begin{proof}This follows from Lemma 1.5 and the definition of $H_X$.
\end{proof}

\begin{lemma}Fix $X\in M_{n\times 2m+1}(F)$ and let $U=F^n X$.
Consider $F^{2m+1}$ as a symmetric vector space with respect to $w_{2m+1}$.
If $U$ is non--degenerate, then the right stabilizer $H'_X$ of $U$ in $G(m)$ is the stabilizer of an involution of $O_{2m+1}(F)$.
If $U$ is degenerate, then $H'_X$ is contained in a proper parabolic subgroup of $G(m)$.
\end{lemma}

\begin{proof} The argument is the same as the corresponding statements of
Lemma 4.5 of \cite{Sh3}, Lemma 3.6 of \cite{GS1}, and Lemma 2.7 of \cite{GS2}.
\end{proof}

\section{The norm correspondence and the case $n=2m+1$}

\begin{deff}Let \(\{Y^{-1}\}\) be an $\varepsilon$--conjugacy class in ${\cal N}.$
Denote by \(N_\varepsilon(\{Y^{-1}\})\) the conjugacy classes $\{(-1)^n(I-X_1'Y_1^{-1}X_1)\}$ for all solutions
$(X_1,Y_1')$ to (1.1) with $Y_1^{-1}\in \{Y^{-1}\}$. This correspondence is finite to one.
\end{deff}

\begin{prop}Suppose $n < 2m$, and $X\in M_{n\times 2m+1}(F)$.
Fix an invertible $Y$ with $(X,Y)$ a solution to (1.1).
Then $(-1)^n (I-X' Y^{-1} X)$ belongs to a proper parabolic subgroup or a proper centralizer of a singular element of $O_{2m+1}(F)$.
Furthermore, $\{N_\varepsilon(\{Y^{-1}\})\}$ is never regular elliptic.
\end{prop}

\begin{proof}Since $n < 2m$, and $X\in M_{n\times 2m+1}$, we have $F^n X \subsetneq F^{2m+1}$, and thus by Lemma 1.5(b) and Lemma 1.7, we have the first result.
Since $(-1)^n (I-X' Y^{-1} X)$ has at least $2m+1-n$ eigenvalues equal to $\pm 1$, the second result follows.
\end{proof}

\begin{lemma}Suppose $S\in M_{2m+1}(F)$, with $-(I+S)\in G(m)$.
Then there is some $Y\in GL_{2m+1}(F)$ and a projection $X\in M_{2m+1}(F)$ with $S=-X' Y^{-1} X=Y^{-1} X=-X' Y^{-1}$.
\end{lemma}

\begin{proof}
The proof follows the arguments given in Corollary 5.7 of \cite{Sh3}, Lemma 3.8 of \cite{GS1}, and Lemma 2.10 of \cite{GS2}.
\end{proof}

\begin{deff}When $(X,Y)$ satisfies the hypotheses of Lemma 2.2, we say that $\{Y^{-1}\}$ is the {\emph canonical section} of the norm correspondence of $\{Z\}=\{-(I-X' Y^{-1} X)\}$.
\end{deff}

\begin{lemma}If $h\in SO_{2m+1}(F)$, then the dimension of the $h$--fixed subspace in $F^{2m+1}$ is odd.
\end{lemma}

\begin{proof}Let $h\mapsto h_0=\begin{pmatrix} h&0\\ 0&1\end{pmatrix}\in SO_{2m+2}(F)$.
By Lemma 5.8 of
\cite{Sh3}, we know $(F^{2m+2})^{h_0}$ is even dimensional.
Since $\dim (F^{2m+2})^{h_0}=\dim (F^{2m+1})^h+1$, and $\dim (F^{2m+2})^{h_0}\geq 2,$ we have the lemma.
\end{proof}

\begin{corollary} Let $h\in SO_{2m+1}(F)$ be semisimple.
Then $h$ is conjugate to an element of the form $\bmatrix h_{11}&0&h_{12}\\ 0&1&0\\ h_{21}&0&h_{22}\endbmatrix$, with each $h_{ij}\in M_m(F)$.
\end{corollary}

\begin{proof}This follows from $h$ having fixed space of dimension at least one.
\end{proof}

\begin{prop}Let $h\in SO_{2m+1}(F)$ be semisimple.
Then there is a $Y\in GL_{2m+1}(F)$ and an $X\in M_{2m+1}(F)$ so that $(X,Y)$ satisfies (1.1), and $h=-(I-X' Y^{-1} X)$.
\end{prop}

\begin{proof}By Corollary 2.6 we may assume $h=\bmatrix h_{11}&0&h_{12}\\ 0&1&0\\ h_{21}&0&h_{22}\endbmatrix$, with
$h_{ij}\in M_m(F)$. Let $h_0=-\bmatrix h_{11}&h_{12}\\ h_{21}&h_{22}\endbmatrix\in SO_{2m}(F)$.
Then, by 
\cite{Sh3}, there is a $Y_0\in GL_{2m}(F)$ and a projection $X_0\in M_{2m}(F)$ with
\begin{eqnarray*}
I-X'_0 Y_0^{-1} X_0&=&h_0,\text{ and}\\
Y_0+\tilde\varepsilon(Y_0)&=&X_0 X'_0.
\end{eqnarray*}
Let $Y_0=\bmatrix Y_{11}&Y_{12}\\ Y_{21}&Y_{22}\endbmatrix$, $X_0=\bmatrix X_{11}&X_{12}\\ X_{21}&X_{22}\endbmatrix$,
and $Y_0^{-1}=\bmatrix Z_{11}&Z_{12}\\ Z_{21}&Z_{22}\endbmatrix$, $X_{ij},\ Y_{ij},\ Z_{ij}\in M_m(F)$. Note that
$X'_0=-\bmatrix \tilde\varepsilon(X_{22})&\tilde\varepsilon(X_{12})\\
\tilde\varepsilon(X_{21})&\tilde\varepsilon(X_{11})\endbmatrix$ and $\tilde\varepsilon(Y_0)=\bmatrix
\tilde\varepsilon(Y_{22})&\tilde\varepsilon(Y_{12})\\ \tilde\varepsilon(Y_{22})&\tilde\varepsilon(Y_{11}) \endbmatrix$.
So,
\begin{eqnarray*}
Y_0+\tilde\varepsilon(Y_0)&=&\bmatrix Y_{11}+\tilde\varepsilon(Y_{22})&
Y_{12}+\tilde\varepsilon(Y_{12})\\ Y_{21}+\tilde\varepsilon(Y_{21})&
Y_{22}+\tilde\varepsilon(Y_{11})\endbmatrix\\
=X_0 X'_0&=& - \bmatrix X_{11}\tilde\varepsilon(X_{22})+X_{12}\tilde\varepsilon(X_{21})&X_{11}\tilde\varepsilon(X_{12})+X_{12}\tilde\varepsilon(X_{11})\\
X_{21}\tilde\varepsilon(X_{22})+X_{22}\tilde\varepsilon(X_{21})&X_{21}\tilde\varepsilon(X_{12})+X_{22}\tilde\varepsilon(X_{11})\endbmatrix.
\end{eqnarray*}
Now, let $Y=\bmatrix Y_{11}&0&Y_{12}\\ 0&-1/2&0\\ Y_{21}&0&Y_{22}\endbmatrix$ and $X=\bmatrix X_{11}&0&X_{12}\\
0&1&0\\ X_{21}&0&X_{22}\endbmatrix$. Then
\begin{eqnarray*}
\tilde\varepsilon(Y)&=&
\bmatrix\varepsilon (Y_{22})&0&\tilde\varepsilon(Y_{12})\\ 0&-1/2&0\\ \tilde\varepsilon(Y_{21})&0&\tilde\varepsilon(Y_{11})\endbmatrix,
\text{ and }\\
X'&=& - \bmatrix \tilde\varepsilon(X_{22})&0&\tilde\varepsilon(X_{12})\\ 0&1&0\\ \tilde\varepsilon(X_{21})&0&\tilde\varepsilon(X_{11}) \endbmatrix.
\end{eqnarray*}
Thus, direct computation shows $(X,Y)$ satisfies (1.1).
Furthermore,
\begin{eqnarray*}
I&-&X' Y^{-1} X=\\
I&+&\bmatrix \tilde\varepsilon(X_{22})&0&\tilde\varepsilon(X_{12})\\ 0&1&0\\ \tilde\varepsilon(X_{21})&0&\tilde\varepsilon(X_{11})\endbmatrix \bmatrix Z_{11}&0&Z_{12}\\ 0&-2&0\\ Z_{21}&0&Z_{22}\endbmatrix \bmatrix X_{11}&0&X_{12}\\ 0&1&0\\ X_{21}&0&X_{22}\endbmatrix\\
=&-&\bmatrix h_{11}&0&h_{12}\\ 0&1&0\\ h_{21}&0&h_{22}\endbmatrix, \text{ so }-(I-X' Y^{-1} X)=h,
\end{eqnarray*}
as claimed.
\end{proof}

\begin{corollary}The elements $X\in M_{2m+1}(F)$ and $Y\in GL_{2m+1}(F)$ can be chosen so that $X$ is a projection, and $X' Y^{-1}X=X' Y^{-1}= -Y^{-1} X$.
\end{corollary}

\begin{proof}We know from 
%cite{shahidi notion of norm}
{\cite {Sh3}} that the elements $X_0,Y_0$ in the proof may be chosen so that $X_0$ is a projection, and
$X'_0 Y_0^{-1} X_0=X'_0 Y_0^{-1}=-Y_0^{-1} X_0$. Then the specified elements $X,Y$ of the proof of Proposition 2.7 satisfies
the claim.
\end{proof}

\begin{lemma}Let $n=2m+1$.
Suppose $\{Y\}\in\cal N$ and $X$ is a projection satisfying (1.1) with $Y$.
Then $-(I-X' Y^{-1} X)$ determines the semisimple part of the conjugacy class of $\varepsilon(Y^{-1})Y^{-1}$, uniquely.
\end{lemma}

\begin{proof}If $v$ is in the left image of $X$, then by Lemma 1.5(b)
$$
v\varepsilon (Y^{-1}) Y^{-1} X=v (-(I-X' Y^{-1} X)),
$$
While if $vX=0$ we have $v(Y+\tilde\varepsilon(Y))=vXX'=0$, and hence $v\varepsilon(Y^{-1})Y^{-1}=-v$.
Thus, the matrix of $\tilde\varepsilon(Y^{-1})Y^{-1}$ with respect to a basis respecting the decomposition $F^{2m+1}=Im X\oplus\text{ Ker }X$, is
$\bmatrix-(I-X' Y^{-1} X)|_{Im X}&*\\ 0&-I\endbmatrix$, proving the lemma.
\end{proof}

\begin{lemma}Suppose $F$ is algebraically closed, and $n=2m+1$.
Let $\{Y\}\in\cal N$ be $\varepsilon$--semisimple with $Y$ in an $\varepsilon$--stable Cartan subgroup of $GL_n(F)$.
Then, there is an $X\in M_n(F)$ satisfying (1.1) with $Y$ so that $-(I-X' Y^{-1} X)$ is semisimple in $G(m)$.
Moreover, $\varepsilon(Y^{-1}) Y^{-1}$ is in $G(m)$ and every $GL_n(F)$--conjugate of $\varepsilon(Y^{-1})Y^{-1}$ belongs to the image of $\{Y^{-1}\}$ under $N_\varepsilon$.
\end{lemma}

\begin{proof}Choose $h\in G(m)$ so that $Y_1=h Yh^{-1}=\diag\{a_1,a_2,\ldots,a_{2m+1}\}$.
Then 
$$Y_1+\tilde\varepsilon(Y_1)=\diag\{a_1+a_{2m+1},a_2+a_{2m},\ldots,a_m+a_{m+2}, 2a_{m+1},a_{m}+ a_{m+2},\ldots,a_1+a_{2m+1}\}.$$
Set $X_1=i \diag \{a_1+a_{2m+1},a_2+a_{2m},\ldots,a_m+a_{m+2},
\sqrt{2a_{m+1}},1,1,\ldots,1\}$, with $i=\sqrt{-1}$.
Then
$$
X'_1=-i \text{ diag }\{1,1,\ldots,1,\ \sqrt{2a_{m+1}},\ a_m+a_{m+2},\ldots,a_1+a_{2m+1}\}.
$$
So
$$
X_1 X'_1=Y_1+\tilde\varepsilon(Y_1),$$
and
$$
I-X'_1 Y_1^{-1} X_1=-\text{diag}\{a_1^{-1} a_{2m+1}, a_2^{-1} a_{2m},\ldots,a_m^{-1} a_{m+2},-1,a_m a_{m+2}^{-1},\ldots,a_1 a_{2m+1}^{-1}\}.
$$
So $-(I-X' Y^{-1} X)=\text{ diag}\{b_1,b_2,\ldots,b_m,1,b_m^{-1},\ldots,b_1^{-1}\}$, with $b_i=a_i^{-1} a_{2m+2-i}$, $i=1,2,\ldots,m$, is semisimple in $G(m)$.

Let $X=hX_1$.
Then we have 
\begin{eqnarray*}
I-X' Y^{-1} X&=&I-(X'_1\varepsilon(h)^{-1}) h Y_1^{-1} h^{-1} (h X_1)\\
&=&I-X'_1 Y_1^{-1} X_1\text{ is semisimple},
\end{eqnarray*} 
proving the first statement.

We note that $Y_1^{-1}\varepsilon(Y_1^{-1})=\text{ diag}\{b_1,b_2,\ldots,b_m,1,b_m^{-1},\ldots,b_1^{-1}\}$, is semisimple in $G(m)$.
Now choose $Y_2$, an $\varepsilon$--conjugate of $Y$ for which there is a projection $X_2$ satisfying (1.1) with $Y_2$ so that $I-X' Y^{-1} X=I-X'_2 Y_2^{-1} X_2$ (Corollary 2.8).

Then $\varepsilon(Y_2^{-1})Y_2^{-1}$ is conjugate to $\varepsilon(Y^{-1})Y^{-1}$ and, by the proof of Lemma 2.9, $\varepsilon(Y_2^{-1})Y_2^{-1}$ has matrix
$\begin{pmatrix} -(I-X'_2 Y_2^{-1} X_2)|_{Im X_2}&*\\
0&-I\end{pmatrix}$ with respect to some basis.
As $Y$ belongs to an $\varepsilon$--stable Cartan, $Y^{-1}\varepsilon(Y^{-1})=\varepsilon(Y^{-1})Y^{-1}$ and thus $Y_2^{-1}\varepsilon(Y_2^{-1})$ and $\varepsilon(Y_2^{-1})Y_2^{-1}$ have the same eigenvalues.
Thus, the eigenvalues of $-(I-X'Y^{-1}X)=-(I-X'_2 Y_2^{-1}X_2)$ which are not equal to $-1$ are among those of $\varepsilon(Y_2^{-1})Y_2^{-1}$.
Therefore, $\varepsilon(Y^{-1})Y^{-1}$ and $-(I-X'Y^{-1}X)$ are $GL_n(F)$--conjugate.
\end{proof}

\begin{lemma}Let $n=2m+1$ and suppose $F$ is not necessarily algebraically closed.
\begin{description}
\item[(a)]Suppose $\{Y\}\in\cal N$ is $\varepsilon$--semisimple.
If $Y$ satisfies (1.1) with some $X$ for which $-(I-X' Y^{-1} X)\in G(m)$ is semisimple, then all other conclusions of Lemma 2.10 hold.

\item[(b)]The semisimple part of every conjugacy class in ${\cal N}_\varepsilon(\{Y^{-1}\})$ is $GL_n(F)$--conjugate to $\{\varepsilon(Y^{-1})Y^{-1}\}$.
\end{description}
\end{lemma}

\begin{lemma}Let ${\cal C}_{ss}$ be the collection of semisimple conjugacy classes in $G(m)$, and $\cal C'$ the $\varepsilon$--conjugacy classes in $GL_{2m+1}(F)$.
Then $N_\varepsilon: {\cal C}_{ss}\longrightarrow\cal C'$ is continuous.\end{lemma}
(This follows in the same manner as Proposition 5.9 of \cite{Sh3},
Lemma 3.16 of \cite{GS1}, and Lemma 2.21 of \cite{GS2}.)

For any $n$, we let $\theta^*$ be the automorphism of $GL_n$ defined by $\theta^*(g)=u_n\ ^t\!g^{-1} u_n^{-1}$,
with
$$
u_n=\begin{pmatrix} &&&&&-1\\ &&&&1\\ &&&-1\\&&.\\&.\\.\end{pmatrix}.
$$
Note that $u_n=g_\theta w_n$, with $g_\theta=\text{ diag}\{-1,1,-1,\ldots\}$.
So $\theta^*=\text{Int}(g_\theta)\varepsilon$.

\begin{lemma}$\theta^*$ preserves a splitting of $GL_n$ over $F$.
\end{lemma}

\begin{proof}Let $\bold B'_0$ be the Borel subgroup of upper triangular matrices in $GL_n$ and $\bold
T'_0$ the maximal split torus of diagonal elements. Let $X_i(t)=t E_{i,i+1}$, with $E_{i,i+1}$ the standard
basis element of $M_n(F)$. Then $(\uB'_0,\uT'_0,\{X_i\}^{n-1}_{i=1})$ is a splitting of $GL_n$ over $F$.
It is a straightforward matrix computation that $\theta^*_0(I+X_i)=I+X_{n-i}$, so $\theta^*$ fixes this splitting.
\end{proof}

For convenience, we  denote $GL_n$ by $\bold G',$ and for $Y\in \bold G',$ we set $\bold G'_{\varepsilon,Y}$ to be the
twisted centralizer of $Y$ in $\bold G'.$  That is
$$\bold G'_{\varepsilon, Y}=\{g\in\bold G'\,|g^{-1}Y\varepsilon(g)=Y\}.$$
If $Y\in G'=\bold G'(F),$ then we let $G'_{\varepsilon,Y}=\bold G'_{\varepsilon,Y}(F).$

We note the following result, which is implicit on pg. 273 of \cite{GS1}, whose proof is the same as the similar statement on pp 258-259 of \cite{GS2}.  

\begin{lemma}We have $G'_{\varepsilon,\varepsilon(Y)}=G'_{\varepsilon,Y^{-1}}=
\varepsilon(G'_{\varepsilon,Y})$.
\end{lemma}

\begin{lemma}Assume $n=2m+1$, that $(X,Y)$ satisfies (1.1), and $Z=-(I-X' Y^{-1} X)$.
Suppose $g\in G'_{\varepsilon,Y}(F)$ and there is an $h\in G(m)$ with $gX=Xh$.
Then $h$, whose class modulo the right stabilizer of $X$ is uniquely determined, belongs to $G_Z(F)$.
Conversely, suppose $h\in G_Z(F)$, and $(X,Y)$ gives the canonical section over $Z$.
If there is some $g\in G'(F)$ with $gX=Xh$, then there is such a $g\in G'_{\varepsilon,Y}(F)$.
\end{lemma}

\begin{proof} The proof here is similar to that of Lemma 3.17 of \cite{GS1}, or Lemma 2.19 of \cite{GS2}.
\end{proof}

Suppose $F=\overline F$.
Let $\uT'_0$ be the maximal torus of diagonal elements of $G'$.
Define $\overline N_{\theta^*}(Y)=Y\theta^* (Y)$.
If $Y=\text{ diag}\{a_1,a_2,\ldots,a_{2m+1}\}$, then $\overline N_{\theta^*}(Y)=\text{ diag}\{a_1 a_{2m+1}^{-1},a_2 a_{2m}^{-1},\ldots,a_m a_{m+2}^{-1},1,a_m^{-1} a_{m+2},\ldots,a_1^{-1} a_{2m+1}\}$.

Therefore, we have $\ker\overline N_{\theta^*}=\{\text{diag}\{a_1,a_2,\ldots,a_m,a_{m+1},a_m,a_{m-1},
\ldots,a_1\}\}$.
Let $Y_0=\text{ diag}\{a_1,a_2,\ldots,a_m,\sqrt{a_{m+1}},1,1,\ldots\}$.
Then $Y_0\theta^* (Y_0^{-1})=(I-\theta^*)(Y_0)=\text{ diag}\{a_1,a_2,\ldots,a_m,a_{m+1},a_m,\ldots,a_1\}$.
Thus $\ker \overline N_{\theta^*}=(I-\theta^*) \uT'_0$.
Now suppose $F$ is not necessarily algebraically closed.
Let $\uT_H$ be a Cartan subgroup of $\uG(m)$, defined over $F$.
Choose a $\theta^*$--stable pair $(\uB',\uT')$, of $GL_{2m+1}$ with $\uT'$ defined over $F$ such that there is an isomorphism $\uT_H \xrightarrow{\sim} \uT'_{\theta^*}$ defined over $F$.

\begin{lemma}When $n=2m+1$, the map $Y\mapsto Y\theta^*(Y)$ from $\uT'$ to $\uT'$ has $(\uT')^{\theta^*}=\{t|\theta^*(t)=t\}$ as its image, can be identified with the projection of $\uT'$ onto $\uT'_{\theta^*}$ and is defined over $F$.
\end{lemma}

\begin{prop}Suppose $n=2m+1$.
The norm correspondence $N_\varepsilon$ agrees with the negative of the norm map of Kottwitz and Shelstad \cite{KS} on the intersection of $\cal N$ with the strongly $\varepsilon$--regular $\varepsilon$--semisimple conjugacy classes in $GL_n$.
\end{prop}

\section{The cases $n=2m$, and $n>2m+1$}

In this section we extend the results of Section 2, first to $n=2m$, and then to $n > 2m+1$.
We begin by proving the $\varepsilon$--norm is surjective on the semisimple classes, when $n=2m$.
For now we always assume $n=2m$.

\begin{lemma}Let $h\in G(m)$ be semisimple.
Then there is a $Y\in GL_n(F)$, and an $X\in M_{n\times 2m+1}(F)$ with $Y+\tilde\varepsilon(Y)=XX'$, so that $\{h\}=\{I-X'Y^{-1}X\}$.
\end{lemma}

\begin{proof}As in the proof of Proposition 2.7, we may assume $h=\bmatrix h_{11}&0&h_{12}\\ 0&1&0\\ h_{21}&0&h_{22}\endbmatrix$, with $h_{ij}\in M_m(F)$.
Let $h_0=\bmatrix h_{11}&h_{12}\\ h_{21}&h_{22}\endbmatrix$ in $SO_{2m}(F)$.
Pick $Y=\bmatrix Y_{11}&Y_{12}\\ Y_{21}&Y_{22}\endbmatrix$ and $X_0=[X_1\ X_2]$ as in Proposition 5.9 of {\cite {Sh3}},
with
$Y_i\in M_m(F)$, and $X_1,X_2\in M_{n\times m}(F)$. That is $h_0=I-X'_0 Y^{-1} X_0$.
Now set $X=[X_1\ 0\ X_2]\in M_{n\times 2m+1}(F)$.
Note that $X'=\bmatrix X'_2\\ 0\\ X'_1\endbmatrix$ and
$X X'=X_0 X'_0=Y+\tilde\varepsilon(Y)$.
Furthermore,
\begin{eqnarray*}
&&I-X' Y^{-1} X=I-\bmatrix X'_2\\ 0\\ X'_1\endbmatrix Y^{-1}[X_1\ 0\ X_1]\\
&=& I-\bmatrix X'_2 Y^{-1} X_1&0&X'_2 Y^{-1}X_2\\ 0&0&0\\ x'_1 Y^{-1} X_1&0&X'_1 Y^{-1} X_2\endbmatrix,
\end{eqnarray*}
and direct comparison, we see $I-X' Y^{-1} X=h$.
\end{proof}

\begin{lemma}Suppose $F$ is algebraically closed, and $n=2m$.
Let $Y\in GL_n(F)$ be $\varepsilon$--semisimple and in an $\varepsilon$--stable Cartan subgroup.
Then there is an $X\in M_{2\times 2m+1}(F)$ satisfying (1.1) with $Y$ for which $(I-X' Y^{-1} X)$ is semisimple in $G(m)$.
\end{lemma}

\begin{proof}Let $\uT'$ be an $\varepsilon$--stable Cartan containing $Y$.
Let $\uT'_1=\left\{\begin{pmatrix} A_{11}&0&A_{12}\\ 0&1&0\\ A_{21}&0&A_{22}\end{pmatrix}\bigg| \begin{pmatrix} A_{11}&A_{12}\\ A_{21}&A_{22}\end{pmatrix}\in\uT' \right\}$ with each $A_{ij}\in M_m(F)$.
Let $\varphi:\uT'\to \uT'_1$ be the obvious map.
Then $\tilde\varepsilon(\varphi(A))=\varphi(\tilde\varepsilon(A))$, and thus $\uT'_1$ is $\tilde\varepsilon$--stable,
and hence $\varepsilon$--stable. By Proposition 2.7 there is an $X_1\in M_{2m+1}(F)$ satisfying (1.1) with $\varphi(Y)$
so that $I-X'_1\varphi(Y)^{-1}X_1$ is semisimple. Then, the proof of Lemma 3.1 shows there is an $X\in M_{n\times
2m+1}(F)$ satisfying (1.1) with $Y$ so that $I-X'Y^{-1} X=I-X'\varphi(Y)^{-1}X$ is semisimple.
\end{proof}

\begin{corollary}
\begin{description}
\item[(a)]If $F$ is algebraically closed and $\varphi$ is as in the proof of the Lemma 3.2, then every $GL_{2m+1}(F)$
conjugate of $\varphi[\varepsilon(Y^{-1})Y^{-1}]$ belongs to the image of $\{Y^{-1}\}$ under $N_\varepsilon$.

\item[(b)]Suppose $F$ is not necessarily algebraically closed.
If there is an $X\in M_{n\times 2m+1}(F)$ satisfying (1.1) with $Y$ for which $I-X' Y^{-1} X$ is semisimple, then all of the above conclusions hold.
In particular the semisimple part of the conjugacy classes in $N_\varepsilon(\{Y^{-1}\})$ is $GL_{2m+1}(F)$--conjugate to $\varphi(-\varepsilon(Y^{-1})Y^{-1})$.
\end{description}
\end{corollary}

Here we again choose a Cartan subgroup $\uT_H$ of $\uG(m)$ defined over $F$.
By Lemma 2.16 and composing with $\varphi$, or by the above (or by
{\cite {Sh3}}) we can choose a $\theta^*$--stable pair $(\uB',\uT')$ of $GL_n(F)$ with $T_H\xrightarrow{\sim}
T'_{\theta^*}$ defined over $F$.

\begin{lemma}(a) The map $Y\mapsto Y\theta^*(Y)$ from $\uT'$ to $\uT'$ has $(\uT')^{\theta^*}$ as its image, and can be identified with the projection of $\uT'$ onto $T'_{\theta^*}$.

(b) $N_\varepsilon^{-1}:{\cal C}_{ss}\longrightarrow \cal C'$ is continuous.
\end{lemma}

\begin{proof}
\begin{description}
\item[(a)]This was shown in {\cite {Sh3}}.

\item[(b)]This follows from Lemma 2.12 and composing with $\varphi$.
\end{description}
\end{proof}

\begin{prop}The map $\varphi\circ N_{\theta_*}$ agrees with the norm map of Kottwitz and Shelstad.
\end{prop}

We now consider the case $n > 2m+1$, and $n$ of any parity.

\begin{lemma}If $n > 2m+1$, then the image of the $\varepsilon$--norm map $N_\varepsilon :\cal N\to\cal C$ contains all the semisimple classes.
\end{lemma}

\begin{proof}If $n$ is odd, we inject $SO_{2m+1}\hookrightarrow SO_n$ by $h\mapsto\begin{pmatrix} I_j\\ &h\\ &&I_j\end{pmatrix}=h_1$, with $j={n-2m+1\over 2}$.
Then the argument of Corollary 3.3 of {\cite {GS1}} applies.
If $n$ is even, then we embed $\bold G(m)\hookrightarrow \bold G(n)$ as above.
We can choose $Y\in GL_n(F)$ and $X_0\in M_{n,n+1}(F)$ so that $I-X'_0 Y^{-1} X_0=h_1$.
By the proof of Lemma 3.2, we can take $X_0=[X_1\ 0\ X_2]$, with $X_i\in M_{n\over 2}$.
Take $X_1=[A_{11}\ A_{12}]$, and $X_2=[A_{21}\ A_{22}]$.
Now, taking
$$
X=[A_{12}\ 0\ A_{21}]\in M_{n\times 2m+1}(F)
$$
and comparing the block forms of $I-X'_0 Y^{-1} X_0$ and $I-X' Y^{-1} X$, we see that $I-X' Y^{-1} X=h$.
\end{proof}

\section{Poles of intertwining operators}
We now use the results of \S2 and \S3 to determine the poles of the operators $A(s,\tau'\otimes\tau,w_0)$, and compute their residues.

\begin{prop}Suppose $n=2m$ or $n=2m+1$.
Further suppose that the $\varepsilon$--conjugacy class $\{Y^{-1}\}$ is $\varepsilon$--regular.
Then $N_\varepsilon (\{Y^{-1}\})$ consists of a single regular semisimple class in $G=G(m)$.
If we assume that $Y^{-1}$ and $\varepsilon(Y^{-1})$ commute (i.e.~$-Y^{-1}\varepsilon(Y^{-1})$
$(n=2m+1)$ or $-\varphi(Y^{-1}\varepsilon(Y^{-1}))(n=2m)$), is in $G$) then the converse holds, i.e., if
$N_{\varepsilon}(\{Y^{-1}\})$ is regular semisimple, then $\{Y^{-1}\}$ is $\varepsilon$--regular and $\varepsilon$--semisimple.
\end{prop}

\begin{proof}For $n=2m$ the statements of this proposition follow from Proposition 4,1 of {\cite {GS1}}, or {\cite {Sh3}}, by
composing with the embedding $\varphi\colon SO_{2m}\hookrightarrow SO_{2m+1}$. Thus, assume $n=2m+1$. Then we know,
from the proof of Proposition 2.7, that we may choose $Y_1$,
$\varepsilon$--conjugate to $Y$, and $Y_0\in GL_{2m}(F)$, so that
$-Y_1^{-1}\varepsilon(Y_1)^{-1}=-\varphi(Y_0^{-1}\varepsilon(Y_0^{-1}))$ and hence, by the above, we have
$Y_1^{-1}\varepsilon(Y_1^{-1})$ is regular and semisimple. Thus, so is 
$Y^{-1}\varepsilon(Y^{-1})$ which is
$GL_{2m+1}(F)$--conjugate to $Y_1^{-1}\varepsilon(Y^{-1}_1)$.  Choosing $Y_2,X_2$ with $X_2$ a projection and
$Y_2^{-1}$ which is $\varepsilon$--conjugate to $Y^{-1}$ (and $(X_2,Y_2)$ a solution of (1.1)), we may assume $I-X_2'
Y_2^{-1} X_2=I-X' Y^{-1} X$. As in Lemma 5.10 of 
\cite{Sh3}, we see that the eigenvalues of
$Y_2^{-1}\varepsilon(Y_2^{-1})$ different from 1 are among those of the semisimple part of $I-X_2' Y_2^{-1} X_2$. But
$Y^{-1}\varepsilon(Y^{-1})$ is
$GL_n(F)$--conjugate to $\varepsilon(Y^{-1})Y^{-1}$, so the eigenvalues of $Y_2^{-1}\varepsilon(Y_2^{-1})$ and those of
$\varepsilon(Y_2^{-1}) Y_2^{-1}$ are the same. Thus, by Lemma 1.5, 
the eigenvalues of the semisimple part of $I-X'_2 Y_2^{-1}
X_2$ are among those of $-\varepsilon(Y_2^{-1})Y_2^{-1}$. Therefore, 
the semisimple parts of $I-X' Y^{-1} X$ and
$-\varepsilon(Y^{-1})Y^{-1}$ are
$GL_n(\overline F)$--conjugate. Since $-\varepsilon(Y^{-1})Y^{-1}$ is $GL_n(\overline F)$--conjugate to a regular
element in
$G(\overline F),$ we see $I-X' Y^{-1} X$ must also be regular and semisimple. Now suppose
$Y+\tilde{\varepsilon}(Y)=XX'$, and
$Y^{-1}\varepsilon(Y^{-1})\in G(F)=SO_{2n+1} (F)$, and suppose $N_{\varepsilon}(\{Y^{-1}\})$ contains a regular semisimple class $\{I-X'
Y^{-1} X\}$. Then by Lemma 1.5(a) and Lemma 2.9, the conjugacy class of $-(I-X' Y^{-1} X)$ is completely determined by
the semisimple part of $-\varepsilon(Y^{-1}) Y^{-1}\in G(F)$. Conversely, Lemma 2.9 shows that $-(I-X' Y^{-1} X)$
determines the semisimple part of the conjugacy class of $Y^{-1}\varepsilon(Y^{-1})$, uniquely. Since $\{I-X' Y^{-1} X\}$ is
regular semisimple, and $Y^{-1}\varepsilon (Y^{-1})\in G(F)$, then we must have $Y^{-1}\varepsilon(Y^{-1})$ is regular
and semisimple. Therefore, $\{Y^{-1}\}$ must be $\varepsilon$--semisimple, as in the proof of Proposition 4.1 of {\cite {GS1}}.
\end{proof}

\begin{corollary} For $n > 2m+1$, and almost all regular elliptic conjugacy classes $\{h\}\in G(m)$, the collection $N_{\varepsilon}^{-1}(\{h^{-1}\})$ of $\varepsilon$--conjugacy classes parameterized by $\{h\}$ is a unique $\varepsilon$--regular, $\varepsilon$--conjugacy class in $GL_{2m+1}(\overline F)$.
\end{corollary}

\begin{proof} By the proof of Corollary 2.2 of {\cite {GS1}}, and
the proof of Proposition 2.7, for almost all regular classes in $G(m)$, there is a choice of $Y_2\in GL_{2m+1}(F)$
satisfying (1.1) with $X_2=I_{2m+1}$, and
$I-X'_2 Y^{-1}_2 X_2\in \{h\}.$  Thus, $Y_2+\tilde\varepsilon
(Y_2)=I'_{2m+1}=-w_{2m+1}w_{2m+1}=-I_{2m+1}$.
By Proposition 4.1 the $\varepsilon$--conjugacy class of $Y_2$ is $\varepsilon$--regular, and uniquely determined by
$h$. Now if $n > 2m+1$ is odd, we take $X=\begin{pmatrix} 0\\ I_{2m+1}\\ 0\end{pmatrix} \in M_{n\times 2m+1}(F)$ so
$X'=(0,-I_{2m+1},0)$ and
$$
XX'=\begin{pmatrix} 0&0&0\\ 0&-I_{2m+1}&0\\ 0&0&0\end{pmatrix}.
$$
Taking
$Y=\begin{pmatrix} I\\ &Y_2\\ &&-I\end{pmatrix},$ so 
$\tilde\varepsilon(Y)=
\begin{pmatrix} -I\\ &\tilde\varepsilon(Y_2)\\ &&I\end{pmatrix},$
we have $Y+\tilde\varepsilon(Y)=XX'$ and 
$$
\aligned
I_{2m+1}-X' Y^{-1} X&=
I_{2m+1}-(0,-I_{2m+1},0)\begin{pmatrix} I\\ &Y_2^{-1}\\ &&I\end{pmatrix}
\begin{pmatrix} 0\\ I_{2m}\\ 0\end{pmatrix}\\
&=I-\begin{pmatrix} 0&-Y_2^{-1}&0\end{pmatrix}\begin{pmatrix} 0\\ I_{2m}\\ 0\end{pmatrix}
=I+Y_2^{-1}=h.
\endaligned
$$
Thus, we need to check that for almost all $Y_2$, the class of $Y$ satisfying (1.1) is, up to $GL_n(\overline F)$--conjugacy, given as above.

Note that $\text{Ker }XX'=\bigg\{\begin{pmatrix} a_1\\ 0\\ a_2\end{pmatrix}\bigg| a_i\in F^j\bigg\},$ with
$\displaystyle{j=\frac{n-(2m+1)}{2}},$ and $\text{Im }
X'=\bigg\{\begin{pmatrix} 0\\ b\\ 0\end{pmatrix}\bigg| b\in F^{2m+1}\bigg\}$. 
Furthermore, we note that  $Y^{-1}$ (and hence $Y$) acts semisimply
on the image and kernel of
$XX'$. Thus, $Y=\text{diag}(J_1,\, Y_2,\, J_2),$ with  $J_1, J_2\in M_j(F),$ and 
$$
Y+\tilde\varepsilon(Y)=\begin{pmatrix} J_1\\ &Y_2\\ &&J_2\end{pmatrix}+
\begin{pmatrix} \tilde\varepsilon (J_2)\\ &\tilde\varepsilon(Y_2)\\ &&\tilde\varepsilon(J_1)\end{pmatrix}.
$$
Therefore, $\diag(J_1,J_2)$ is $\tilde\varepsilon$--skew symmetric, and $Y_2+\tilde\varepsilon(Y_2)=X_2 X'_2$ and hence up to $GL_n(F)$--conjugacy, $Y$ is as above.
This proves the claim for $n$ odd.

Now suppose $n$ is even.
By \cite{Sh3}, for almost all regular elliptic conjugacy classes $\{h_1\}$ in $SO_{2m}(F)$, $N_\varepsilon(\{h_1\})$ parameterizes a unique $\varepsilon$--conjugacy class $\{Y_1^{-1}\}\in GL_n(F)$.
Then we see that
$h=\begin{pmatrix} h_{11}&0&h_{12}\\ 0&1&0\\ h_{21}&0&h_{22}\end{pmatrix}$
(where $h_1=\begin{pmatrix} h_{11}&h_{12}\\ h_{21}&h_{22}\end{pmatrix}$) parameterizes
$\{ Y_1^{-1}\}=\{\varphi (Y^{-1})\}=\left\{\begin{pmatrix} Y_{11}&0&Y_{12}\\ 0&1&0\\
Y_{21}&0&Y_{22}\end{pmatrix}\right\}\subset G({n\over 2}+1)$, and by the above this is unique for almost all $h_1$. Thus,
there is a unique $\{Y^{-1}\}$ parameterized by
$\{h\}$, for almost all $\{h\}$.
\end{proof}

\begin{lemma}Let $\alpha\in (F^\times)^2$ and choose $\lambda\in F^\times$ with $\lambda^2=\alpha$.
Let $\alpha_0=\diag(\alpha I_m,\lambda,I_m)$.
Then for any $\{\gamma'\}\in N$, we have $N_\varepsilon(\{\alpha\gamma'\})=\alpha_0^{-1} N_{\varepsilon} (\{\gamma'\})\alpha_0$.
\end{lemma}

\begin{proof}Note that if $\alpha^\vee=\alpha I_{2m}$, then $\alpha^\vee=\alpha_0\tilde\varepsilon
(\alpha_0)=\alpha_0\varepsilon(\alpha_0)^{-1}$. Thus, if $Y^{-1}\in \{\gamma'\}$ and $(X,Y)$ satisfies (1.1), then we
have 
$$
\alpha Y+\tilde\varepsilon(\alpha Y)=X \alpha^\vee X'=X\alpha_0\varepsilon(\alpha_0)^{-1} X'=(X\alpha_0)(X\alpha_0)'
$$
so $N_\varepsilon\{(\alpha Y)^{-1}\}=N(\{\alpha^{-1}\gamma'\})=$
$$
\{I-(X\alpha_0)' (\alpha Y)^{-1} (X\alpha_0)\}=\{I-\alpha^{-1}\varepsilon
(\alpha_0)^{-1} (X' Y X)\alpha_0\}.
$$
But $\alpha^{-1}\varepsilon(\alpha_0)^{-1}=\alpha^{-1}
\begin{pmatrix} I_m\\ &\lambda\\ &&\alpha I_m\end{pmatrix}=
\begin{pmatrix} \alpha^{-1} I_m\\ &\alpha^{-1}\lambda\\ &&I_m\end{pmatrix}=
\alpha_0^{-1}$.
Thus, $N(\{\alpha^{-1} Y^{-1}\})=\alpha_0^{-1} N(\{Y^{-1}\})\alpha_0$ and therefore $N(\{\alpha Y^{-1}\})=\alpha_0 N\{(Y^{-1})\}\alpha_0^{-1}$.
\end{proof}

Let ${\cal C}_{ss}$ be the collection of semisimple conjugacy classes in $G=G(m)$.
Suppose $n=2m$.
Let $\uT$ be a Cartan subgroup of $\uG(m)$, defined over $F$.
We may choose a $\theta^*$--stable Cartan $\uT'$ of $GL_n$ and an isomorphism $\uT'_\theta\simeq \varphi (\bold T)$
defined over $F$, as in 
%cite{shahidi notion of norm}
{\cite {Sh3}}. Thus, $\uT'_\theta\simeq \uT$ is defined over $F$ as well, and by
\cite{KS} this isomorphism induces the image map $\sA_{\uG(m)/GL_n}$ between
${\cal C}_{ss}$ and $\theta^*$--semisimple $\theta^*$--conjugacy classes in $GL_n$.
Again by 
\cite{Sh3}, we see that
\begin{center}
{$
\begin{diagram}
\node{\uT'}\arrow{ese,t,1}{h}
\arrow{s,r}{N_{\theta^*}}\\
 \node{\varphi(\uT)\simeq \uT'_{\theta^*}} \node[2]{\uT' g_{\theta^*}}\arrow[2]{w,b}{N_\varepsilon}
\end{diagram}
$}
\end{center}
commutes on strongly $\theta^*$--regular elements of $\uT'$.
So
%\begin{equation}
\begin{center}
{$
\begin{diagram}
\node{\uT'}\arrow{ese,t,1}{h}
\arrow{s,r}{N_{\theta^*}}\\
 \node{\uT^*_{\theta^*}\simeq \uT} \node[2]{\uT' g_{\theta^*}}\arrow[2]{w,b}{N_\varepsilon}
\end{diagram}
$}\qquad\text{also commutes}.
\end{center}
%\end{equation}
In the language of 
\cite{KS} we see that $h=m^{-1}$ when we take $\psi=1$.

Now, if $n=2m+1$, then the map $\uT'_{\theta^*}\simeq \uT$ is again defined over $F$ and so the above diagrams are again commutative.

\begin{lemma}Let $\uT$ be a Cartan subgroup of $\uG(m)$ defined over $F$.
Then there is a $\theta^*$--stable Cartan $\uT'$ of $GL_{2m}$ (respectively $GL_{2m+1}$) such that the diagram
\begin{center}
{$
\begin{diagram}
\node{\uT'}\arrow{ese,t,1}{m^{-1}}
\arrow{s,l}{N_{\theta^*}}\\
 \node{\uT\simeq \uT_{\theta^*}'} \node[2]{\uT' g_{\theta^*}}\arrow[2]{w,b}{N_\varepsilon}
\end{diagram}
$}
\end{center}
Commutes up to a sign on all $\theta^*$--strongly regular $\theta^*$--semisimple elements of $\uT'(F)$.
Furthermore $\uT\simeq \uT'_{\theta^*}$ induces the image map
$\sA_{\uG/\uG'}$.
If $\delta^*\in\uT'$ is $\theta^*$--strongly regular, then $\operatorname{Cent}_{\theta^*}$
$(\delta^*,\uG')\simeq(\uT')^{\theta^*}$.
\end{lemma}

\begin{proof}All statements follow from the above observations, except the last.
This follows from Lemma 4.4 of 
\cite{GS1}, or the argument therein as appropriate.
\end{proof}

In order to compute the residue of the intertwining operators, we will need to integrate over twisted conjugacy classes in $\sN$.
We have seen that, up to a set of measure zero, these are parameterized by regular semisimple classes in $G$, i.e.~${\cal C}_{ss}$.
Fix a representative $\uT$ for each conjugacy class of Cartan subgroups of $\uG$ defined over $F$.
Fix $d\gamma$ to be a Haar measure on $T=\uT(F)$.
Then by Lemmas 2.16 and 3.4, along with Lemma 4.4, the Jacobian of the open inversion of pg.~227 of 
\cite{A}, with the measure
$|W(T)|^{-1}|D_{\theta^*}(\gamma')|d\gamma$, as $\uT$ ranges over the conjugacy classes of Cartans, provides a measure on $\sN$ (where $\{\gamma\}\in N_{\theta^*}(\{\gamma'\})$, for each regular $\{\gamma' g_{\theta^*}\}$ in $\sN$).  Here
$$
D_{\theta^*}(\gamma')=\det (\operatorname{Ad} (\gamma')\circ\theta^*-1)|_{G/G_{\theta^*},\gamma'},
$$
is as given in \cite{KS}.
As in previous cases, the  constant $|W(T)|^{-1}$ is suggested by the Weyl integration formula.
Now, by Lemma 4.5.A of 
\cite{KS}, the function
$$
\kappa_1 (\{\gamma\}, \{\gamma'\})=|D_{\theta^*} (\gamma')|/|D(\gamma)| 
$$
is bounded on $\{(\{N_{\theta^*} (\{\gamma'\})\},\ \{\gamma'\})\}$.
Assume $\{\gamma\}$ is regular and semisimple.
Define
$$
\kappa (\{\gamma\}, \{\gamma'\})=\begin{cases} \kappa_1 (\{\gamma\},\{\gamma' g^{-1}_{\theta^*}\})&\text{ if }\{\gamma\}\in N_\varepsilon
(\{\gamma'\})\text{ and }\gamma'\text{ is }\varepsilon\text{--regular}\\
0&\text{otherwise}.\end{cases}
$$
For each regular semisimple conjugacy class $\{\gamma\}\in {\sC}_{ss}$, let
$\sA(\{\gamma\})=\{\{\alpha\gamma'\}|\alpha\in (F^{\times})^2\backslash F^\times,\
\{\gamma\}\in N_\varepsilon (\{(\gamma')^{-1}\})\}$.
Now set
$$
\Delta(\{\gamma\},\{\alpha\gamma'\})=
\begin{cases} \omega'(\alpha)\kappa (\{\gamma\},\{\gamma'\})&\text{ if }
\{\gamma\}\in N_\varepsilon (\{(\gamma')^{-1}\});\\ 0&\text{otherwise}. \end{cases}
$$

For $Y\in G'$ and $f'\in C_c^\infty(G'),$ we let $\Phi_{\varepsilon}(Y,f')$ be the associated twisted orbital integral,
$$\Phi_{\varepsilon}=\int\limits_{G'/G'_{\varepsilon,Y}}f(g^{-1}Y\varepsilon(g))\,dg.$$ 
Similarly, for $\gamma\in G,$ and $f\in C^\infty(G),$ we let $$\Phi(\gamma,f)=\int\limits_{G/G_\gamma}f(g^{-1}\gamma g)\,dg$$
be the associated orbital integral.
We also let 
$$\Phi_\varepsilon(\sA(\{\gamma\}),f')=
\sum\limits_{\{\gamma'\}\in\sA(\{\gamma\})} \Delta(\{\gamma\},
\{\gamma'\})\Phi_\varepsilon(\gamma',f'),$$ for any $f\in C^\infty(G').$

The residue of $A(s,\tau'\otimes\tau,w_0)$ at $s=0$ will decompose into two parts.
As in 
\cite{GS1, GS2}, the main part of this residue will come from
regular elliptic elements via the Weyl integration formula applied to the class function
$\Phi_\varepsilon(\sA(\{\gamma\}),f')$, giving a pairing between characters of representations of $G(m)$ and
$\varepsilon$--twisted characters of $\varepsilon$--stable representations of $GL_n(F)$, and we return to this in
Section 5. To be more precise, the contribution from the regular elliptic classes is of the form
$$
R_G(f_\tau,f')=\sum_{\{\uT_i\}}  \mu(T_i)|W(\bold T_i)|^{-1}\int\limits_{T_i}\Phi(\{\gamma\},f_\tau)\Phi_\varepsilon
(\sA(\{\gamma\}),f')|D(\gamma)| d\gamma,
$$
where $\{\uT_i\}$ runs over the conjugacy classes of elliptic Cartan subgroups of $\uG=\uG(m)$, and $T_i=\bold{T_i}(F)$.
For each $i$, $\mu(T_i)$ is the measure of $T_i$, and due to the definition of $\Delta$ and the fact that the norm is
onto the semisimple classes, we see that this is an integral over $\varepsilon$--conjugacy classes in $\sN$.

We now discuss the convergence of $R_G(f_\tau,f')$.
The steps we follow are analogous to those of 
\cite{GS1}.
A new feature, however is the treatment of a case of the form $M\simeq GL_n\times SO_{2m+1}$, with $n$ even.
Clearly, we need to show that, for any elliptic torus $\bold T$ of $\bold G(m)$, the term
$$
\int\limits_{\bold{T}(F)}\Phi(\gamma,f_\tau)\Phi_\varepsilon(\sA(\{\gamma\}), f')
|D(\gamma)| d\gamma
$$
is convergent.
By Theorem 14 of 
\cite{HC1}, $|D(\gamma)|^{1/2}\Phi(\gamma,f_\tau)$ is bounded on $\bold
T_{{\rm reg}}(F),$ the intersection of $\bold T(F)$ with the regular
set of
$G$. Thus, we need to establish the convergence of
$$
\int\limits_{\bold{T}(F)}\Phi_\varepsilon (\sA(\{\gamma\}),f')|D(\gamma)|^{1/2}d\gamma,
$$
which is then implied by the convergence of
\begin{equation}
\int\limits_{\bold T(F)\atop{N_\varepsilon(\{\gamma'\})=\{\gamma\}}}
|\Phi_\varepsilon(\alpha\gamma',f')|\kappa_1(\{\gamma\},\{\gamma'\})|D(\gamma)|^{1/2}d\gamma,
\end{equation}
for any $\alpha\in (F^\times)^2\backslash F^\times$.

Note that $w_n=g_\varepsilon u_n=u_ng_\varepsilon$, where
\begin{eqnarray*}
g_\varepsilon&=&\text{diag}(-1,1,-1\ldots-1)\text{ if }n\text{ is odd and}\\
g_\varepsilon&=&\text{diag}(1,-1,1,\ldots,-1),\text{ when }n\text{ is even}.
\end{eqnarray*}
Then, (4.1) can be rewritten as
\begin{equation}
\int\limits_{\bold T(F)\atop{N_{\theta^*}(\{\gamma'\})=\{\gamma\}}}|\Phi_{\theta^*}
(\gamma',R_{g_\varepsilon\alpha}f')| |D_{\theta^*}(\gamma')||D(\gamma)|^{-1/2}d\gamma
\end{equation}
As $\Phi_{\theta^*}$ is a tempered distribution, we have $\Phi_{\theta^*}(\gamma',f')|D_{\theta^*}(\gamma)|^{1/2}$ is bounded on the intersection of $\bold T_{\theta^*}(F)$ with the $\theta^*$--regular elements of $GL_n(F)$.
(See 
\cite{C1,HC1}.)
Now, since Lemma 4.5.A of 
\cite{KS} gives the boundedness of $\kappa_1(\{\gamma\},\{\gamma'\})^{1/2}$, we see (4.2) must converge.

In order to resolve (1.3), we integrate first over the orbits of $N$ under $M=G'\times G$.
Note that, for $(g,h)\in M$, $d((g,h)(X,Y)(g,h)^{-1})=d(gXh^{-1},gY\varepsilon(g)^{-1})=|\det g|_F^{\langle
2\rho,\tilde\alpha\rangle} d(X,Y)$. So setting $d^*(X,Y)=d(X,Y)\cdot|\det
Y|_F^{-\langle\rho,\tilde\alpha\rangle}$ we have $d^*(gXh^{-1},gY\varepsilon(g)^{-1})=d^*(X,Y)$. As in 
\cite{GS1}, we write $d^*(X,Y)$ as the product of two measures, $d_1^*(X,Y)$ and $d_2^*(X,Y)$, with the
first giving the integral over the orbit of $(X,Y)$ in $N$, and the second integration over all such orbits. The
discussion of Section 1 shows that $d_2^*(Y^{-1}X,\varepsilon(Y))=d_2^*(X,Y)$. As $d^*(X,Y)$ and $d_2^*(X,Y)$ are
$M$--invariant, so is $d_1^*(X,Y)$.

Let $\Delta^\vee{\phantom{}}$ be the stabilizer of $(X,Y)$ under the action of $M$.
Then
$$
\Delta^\vee{}=\{(g,h)\in M|g\in G'_{\varepsilon,Y},\ Xh=gX\}.
$$
Then, direct computation, as in the proof of Lemma 2.15 shows $h\in\bold G_Z(F)$, with $Z=(-1)^n(I-X' Y^{-1} X)$.
For convenience, we abuse notation and view $\Delta^\vee{}$ as a subgroup of both $G'_{\varepsilon,Y}$ and $\bold
G_Z(F)$, by its projection onto its components.

Let $d\delta$ be a measure on $\Delta^\vee{}$, and fix measures $dg,dh$ on $GL_n(F)$ and
$\bold G(F)$ so that $d_1^*(X,Y)$ is the quotient of $dg dh$ by $d\delta$.
Now consider the map $(g,h)\mapsto (gXh^{-1}, gY\varepsilon(g)^{-1})$,
from
$M$ to the orbit of $(X,Y)$. We change the orbit representative, to
$(Y^{-1}X,\varepsilon(Y))$, which has the effect of changing $(g,h)$ to
$(g Y^{-1},h)$. Therefore, $dg$ and $dh$ remain unchanged by this
transformation. Also, acting by $(Y^{-1},I)$ on $(g,h)$ conjugates the
stabilizer $\Delta^\vee{}$ of $(X,Y)$ to $(Y,I)\Delta^\vee{} (Y^{-1},I)$, and thus
this change of variables leaves $d\delta$ unchanged. Now
$d_1^*(Y^{-1}X,\varepsilon(Y))=d_1^*(X,Y)$, and thus
$d^*(X,Y)=d^*(Y^{-1}X,\varepsilon(Y))$. Now we consider (1.3) and make
the above change of variables,
\begin{equation}
\int\limits_{(X,Y)}\psi_{\tau'}(\varepsilon(Y))f_\tau((-1)^n (I-X'Y^{-1}X))|\det
Y|_F^{-s}\xi_L(Y^{-1})\xi_{L'}(Y^{-1}X)d^*(X,Y)\end{equation}
$$=\int\limits_{(X,Y)}\psi_{\tau'}(Y)f_\tau((-1)^n(I-X'Y^{-1}X))\xi_L(\varepsilon(Y)^{-1})\xi_{L'}(X)|\det
Y|^s d^*(X,Y).$$
Let $\omega'$ be the central character of $\tau'$.
Then $\omega'$ is quadratic, and we may choose $f'\in C_c^\infty(G')$ so that
$$
\psi_{\tau'}(g')=\int\limits_{Z(G')}f'(zg')\omega'(z^{-1})d^\times z.
$$
With this substitution, (4.3) becomes
\begin{equation}
\int\limits_{F^\times}\int\limits_{(X,Y)} f'(zY)\omega'(z)f_\tau((-1)^n(I-X' Y^{-1} X))|\det Y|^s\xi_L
(\varepsilon(Y^{-1}))\xi_{L'}(X)d^*(X,Y)d^\times z.
\end{equation}
Now we consider the map from the $M$--orbit of $(X,Y)$ to $G'/\Delta^\vee{}\times\Delta^\vee{}\,\backslash G$, given
by
$$
(gXh^{-1},\ gY\varepsilon(g)^{-1})\mapsto g\Delta'\times\Delta^\vee{} h.
$$
Then the fibers of this map are homeomorphic to $X\Delta^\vee{}$. The contribution from the $M$--orbit of
$(X,Y)$ to (4.4) is then
\begin{eqnarray*}
&&\tilde\psi(s,Z)=\sum_{\alpha\in(F^\times)^2\backslash F^\times}\omega'(\alpha)\int\limits_{g\in
G'/\Delta^\vee{}}\int\limits_{h\in\Delta^\vee{}\backslash G}\int\limits_{X\Delta^\vee{}}f'(\alpha g
Y\varepsilon(g)^{-1})f_\tau(h^{-1}Zh)\\ &&\cdot|\det(g Y\varepsilon(g)^{-1})|^s\, dg \, dh d(X
h_0)\int\limits_{Z(G')}\xi_{\tilde L}(z^{-2} g Y\varepsilon(g)^{-1})\xi_{L'}(z^{-1}gX h_0 h)|\det z|^{-2s}d^\times z,
\end{eqnarray*}
where $\tilde L=\varepsilon(L)$.
Note that here (as in \cite{GS1}) we have suppressed the dependence of $\tilde\psi$ on the parameters $X,Y,f',f_\tau,L$, and $L'$.
Considering $\Delta^\vee{}$ as a subgroup of $G'_{\varepsilon,Y}$ and $G_Z$ (see above), we get
\begin{eqnarray}
&&\tilde\psi(s,Z)=\sum_{\alpha\in(F^\times)^2\backslash F^\times}
\omega'(\alpha)\int\limits_{G'/G'_{\varepsilon,Y}}
\int\limits_{G_Z\backslash G}
\int\limits_{g_0\in G'_{\varepsilon,Y}/\Delta^\vee{}}\int\limits_{X G_Z}\cdot\nonumber\\
&&f'(\alpha g Y\varepsilon(g)^{-1})f_\tau(h^{-1}Zh)|\det
(g Y\varepsilon(g)^{-1})|^s\, dg\, dh dg_0\, dXh_0\,\cdot\nonumber\\
&&\int\limits_{Z(G')}\xi_{\tilde L}(z^{-2} g Y\varepsilon(g)^{-1})
\xi_{L'}(z^{-1}g g_0 X h_0 h)|\det z|^{-2s}d^\times z.
\end{eqnarray}
Let $\displaystyle{\Phi_{s,\varepsilon}(\alpha Y,f')=\int\limits_{G'/G'_{\varepsilon,Y}} f'(\alpha g Y\varepsilon (g)^{-1})|\det (g
Y\varepsilon(g)^{-1}|^s dg}$. Certainly, $\displaystyle{\lim\limits_{s\to 0}\Phi_{s,\varepsilon}(\alpha Y,f')=\Phi_\varepsilon(\alpha
Y,f')=\int\limits_{G'/G'_{\varepsilon,Y}} f'(\alpha g Y\varepsilon(g)^{-1})dg}$ is a twisted orbital integral.

Suppose $n=2m+1$, and $(Y,I)$ is a solution to (1.1).
Then $Z=-(I-I' Y)=-(I+Y^{-1})$.
Note that, if $g\in G'_{\varepsilon,Y}$, then, as $Y+\tilde\varepsilon(Y)=-I$, we have
$-I=g(Y+\tilde\varepsilon(Y))\varepsilon(g)^{-1}=-g\varepsilon(g)^{-1}$, so $g\in G$. Thus, $G'_{\varepsilon,Y}=\bold
G_Z(F)\simeq\Delta^\vee{}$. We then set
\begin{eqnarray}
\psi(s,z)=\sum_\alpha \omega'(\alpha)\int\limits_{G'/G'_{\varepsilon,Y}}
\int\limits_{G_Z\backslash G}\int\limits_{G_Z}f'(\alpha g Y\varepsilon(g)^{-1})f_\tau (h^{-1} Zh)|\det(g
Y\varepsilon(g))|^s\cdot\nonumber\\
\bigg(\int\limits_{Z(G')}\xi_{\tilde L}(z^{-2}g Y\varepsilon(g)^{-1})
\xi_{L'}(z^{-1}g h_0 h)|\det z|^{-2s}d^\times z\bigg)\, dg\, dh\, dh_0.
\end{eqnarray}
In this case, $\psi(s,z)=\tilde\psi(s,Z)$.

\begin{lemma}Let $n=2m+1$.
Assume $Y$ is $\varepsilon$--regular.
Then $\tilde\psi(s,Z)$ converges absolutely for Re $s>0$.
Further, suppose $(Y,I)$ satisfies (2.1), i.e., for almost all twisted conjugacy classes in $\sN$.
Then $\psi(s,Z)$ also converges absolutely for Re $s>0$, and there is a function $E_Z(s)=E(s,z,Y,f_\tau,f',L,L')$,
which as a function of $s$ is entire, such that $\psi(s,Z)=E_Z(s)$ if $Z$ is regular, non--elliptic, and
\begin{eqnarray*}
\psi(s,Z)=E_Z(s)+\qquad\qquad\qquad\qquad\qquad\qquad\qquad\qquad\qquad\\
\sum_{\alpha\in(F^\times)^2\backslash
F^\times}\omega'(\alpha)\cdot\Phi_{\varepsilon,s}(\alpha Y,f')
\Phi(Z,f_\tau)\mu(\bold G_Z(F)) q^{b(Y,Z)^s} L(1,2ns)
\end{eqnarray*}
for Re $s>0$ if $Z$ is regular elliptic.
Here, $b(Y,Z)$ is an integer depending on $f_\tau,f',L,L'$, as well as $Y$ and $Z$.
In particular,
$$
\Res_{s=0}\psi(s,Z)=(2n\log q)^{-1}\sum_\alpha\omega'(\alpha)\Phi_{\varepsilon}(\alpha
Y,f')\Phi(Z,f_\tau)\mu(\bold G_Z(F)),
$$
if $Z$ is regular elliptic, and $\Res\limits_{s=0}\psi(s,Z)=0$ if $Z$ is regular but non--elliptic.
\end{lemma}

\begin{proof}(This is as in \cite{GS1}).
We prove the lemma when $X=I$.
The convergence of $\tilde\psi(s,Z)$ will follow from this.
Recall that we have fixed $L$ and $L'$ to be basic neighborhoods of 0.
Since $f'(\alpha g Y\varepsilon(g)^{-1})\xi_{\tilde L}(z^{-2}gY\varepsilon(g)^{-1})=0$ unless $g Y\varepsilon(g)^{-1}\in
z^2\tilde L\cap\alpha^{-1}\text{ supp}(f')$ for some $\alpha$, the argument used in 
\cite{Sh2, Sh3, G1,G2}
shows that $|\det z|$ is bounded below, and the bound
depends only on $L$ and $f'$. Now the semi-simplicity (respectively $\varepsilon$--semi-simplicity) of $Z$ (resp.~$Y$)
implies the integrand in (4.6) vanishes for $g$ and $h$ outside compact sets $S(g)\subseteq G'/G'_{\varepsilon,Y}$ and
$S(h)\subseteq G_Z\backslash G$. Then $S(g)$ and $S(h)$ depend on $Y,Z,f_\tau$ and $f'$.

We may consider only $h_0$ (in the integrand of (4.6)) within the connected component of $G_Z$, and since $Z$ is
semisimple, we may assume $h_0=(a_1,a_2,\ldots,a_m,1,a_m^{-1},\ldots,a_1^{-1})$ is diagonal.

If $\xi_{L'}(z^{-1}ghh_0h)\not= 0$, then, by the above observations, $z^{-1}h_0$ is contained in a compact set in
$M_n(F)$. Thus, for some $\kappa$, $|z^{-1}a_i|\leq\kappa,\ |z^{-1}a_i^{-1}|\leq\kappa$ and $|z^{-1}|\leq\kappa$, which
we rewrite as $|z|\geq\kappa$, $|z a_i|,\ |z a_i^{-1}|\geq\kappa$. Let $T$ be the compact part of $G_Z$.
Then, by compactness of $S(g),\ S(h)$ we know there is some $\kappa_1,$ such that if $|z^{-1}|,\ |z^{-1}a_i^{-1}|,\
|z^{-1}a_i|\leq\kappa_1$, then $S(g)z^{-1}h_0 TS(h)\subseteq L'$. Clearly, $\kappa_1\geq\kappa$.
So now, fixing such a $z$, we may integrate first over $T$ in the first integral of (4.6), and thus the integral over
$G_Z$ is replaced by $\Phi(Z,f_\tau)$. Note, as $\tau$ is supercuspidal, $\Phi(Z,f_\tau)$ vanishes unless the split
component of $G_Z$ is trivial, i.e., unless $G_Z$ is compact.

Let $\eta$ be the lower bound on $|z|$ given by $f'$ and $L$.
If $\kappa>\eta$ we use $\kappa$ instead of $\eta$.
As $\kappa_1\geq\kappa$, and $|\ \ |$ is discrete, we may assume $\kappa_1/\kappa=q^{-k}$ for some integer $k\geq
0$. Then $|\varpi^{-k}z a_i|\geq\kappa_1$, $|\varpi^{-k}z a_i^{-1}|\geq\kappa$, and $|\varpi^{-k}z|\geq\eta
\kappa_1/\kappa\geq\eta$. Thus, for all $z$ with $|z|\geq\eta\kappa_1/\kappa$, both $\xi_{\tilde L}$ and $\xi_{L'}$ have
value 1, and we can then integrate over all $z=\beta\varpi^{-k}$, with $\beta\in\frak o^\times$ and over all $G_Z$ we
get
\begin{equation}
\sum_{\alpha\in (F^\times)^2\backslash F^\times}\omega'(\alpha)\Phi(Z,f_\tau)\Phi_{\varepsilon,s}(\alpha Y,f')\cdot
\sum_{\kappa\geq\ell}q^{-2 kns}\mu(\kappa)
\end{equation}
where $|\varpi^{-\ell}|=\eta\kappa_1/\kappa$, and
$$
\mu(\kappa)=\mu(T)\prod_i\int\limits_{q^{-(k-d)}\leq|a_i|\leq q^{(k-d)}}d^\times a_i,
$$
with $\mu(T)$ the measure of the compact part $T$ of $G_Z$, with $d$ given by $q^d=\kappa_1$.
As the series converges for Re $s>>0$, so does (4.7).
If $G_Z$ is non--compact (i.e. $Z$ is non--elliptic).
Then $\Phi(Z,f_\tau)=0$, and hence (4.7) vanishes.
All that remains is an integral over $\eta{\kappa_1/\kappa}\geq |z|\geq\eta$, which gives an entire function
$E_Z(s)$.
\end{proof}

\begin{corollary}If $n=2m$, and $Y$ is strongly $\varepsilon$--regular, the statement of Lemma 4.5 holds.
\end{corollary}

\begin{proof}We only need to note the argument of the lemma holds for $X=\bmatrix I_n&0&I_n\endbmatrix$.
If $Y$ is $\varepsilon$--regular, then $G'_{\varepsilon,Y}$ is a torus in $SO_{2m}(F)$.
Let $Z_0$ be the element $I-X'_0 Y^{-1} X_0$ of $SO_{2m}(F)$ given by Lemma 4.5 of \cite{GS1}.
Then $G'_{\varepsilon,Y}=\uH_{Z_0}(F)$, with $\uH=SO_{2m}$.
Now, under the map $\varphi\colon h\mapsto\begin{pmatrix} h_{11}&0&h_{12}\\ 0&1&0\\ h_{21}&0&h_{22}\end{pmatrix}$, we see
that if $g\in G'_{\varepsilon,Y}(F)$, then $\varphi(g)\in \uG_{\varphi(Z_0)}(F)=\uG_Z(F)$, in this case. Thus
$\uG_Z(F)$ contains the torus $\varphi(G'_{\varepsilon,Y})$. But as $Y$ is strongly $\varepsilon$--regular, $Z$ is
strongly regular, and hence $\uG_Z(F)$ is also a torus. As $\varphi(G'_{\varepsilon,Y})$ is a torus of the
appropriate rank, $\uG_Z(F)=\varphi(G'_{\varepsilon,Y})$. Now, the argument of Lemma 4.5 holds verbatim.
\end{proof}

\begin{corollary}Let $\uT$ be a Cartan subgroup of $\uG$.
Denote by $T'$ the subset of regular elements of $T=\uT(F)$.
Let $\Omega$ be a compact set of $T'$.
Then given $f_\tau,f',L$, and $L',\ b=b(Y,Z)$ and $E_Z$ can be chosen independently of $Y$ and $Z$ for all $Z$ in $\Omega$.
\end{corollary}

\begin{proof}The corollary to Lemma 19 of 
\cite{HC1} implies that the compact sets $S(g)$ and $S(h)$ in the proof of Lemma 4.5 can be chosen
independently of
$z$ in $\Omega.$
\end{proof}

In calculating of the residue we now integrate over all $M$--orbits in $N$.
We accomplish this by integrating over all $\varepsilon$--regular $\varepsilon$--conjugacy classes in $\sN$.
We first assume that $n=2m$ or $n=2m+1$.
We then must integrate $\tilde\psi(s,Z)$ under the $M$--orbits in $N$.
We have shown that almost all such orbits are parameterized by $\varepsilon$--regular conjugacy classes in $\sN$.
Thus, by removing a set of measure zero from these orbits, we may integrate $\psi(s,Z)$ over $\varepsilon$--regular $\varepsilon$--conjugacy classes in $\sN$.
Then, by Proposition 4.1, $Z=N_\varepsilon(\{Y^{-1}\})$ is regular and semisimple in $\uG(F)$.

Let $\{\uT_i\}$ be a complete set of representatives for the conjugacy classes of Cartan subgroups of $\uG$ defined over $F$.
We now must integrate over $\sN$, but we may instead use Propositions 4.1, 3.1, and 2.7 to integrate over $\bigcup_i T_i$, using the measures
$$
|W(\bold T_i)|^{-1}\kappa_1(\{\gamma_i\},\{\gamma'_i\})|D(\gamma_i)|
d\gamma_i=|W(\bold T_i)|^{-1}|D_{\theta^*}(\gamma'_i)|d\gamma_i.
$$
Now, suppose $n>2m+1$.
By the proof of Corollary 4.2 we may, for almost all $\{Y\}\in\sN$, choose a representative $\diag(J_1,g,J_2)$, with $g\in GL_{2m+1}(F)$ either $\varepsilon$--regular, or the image under
$\varphi$ of an $\varepsilon$--regular element. Further, $\diag(J_1,J_2)$ is $\tilde\varepsilon$--skew symmetric.
We may also assume
$$
X=\begin{pmatrix} 0_{j\times 2m+1}\\
I_{2m+1}\\
0_{j\times 2m+1}\end{pmatrix},
$$
with $j={n-(2m+1)\over 2}$, if $n$ is odd, and $X$ has a similar form with $I_{2m+1}$ replaced by $\bmatrix I_m&0&I_m\endbmatrix$ for $n$ even.
In this case, outside of a set of measure zero, the classes in $\sN$ form a fiber bundle with finite fibers coming from the twisted conjugacy classes of possible choices of diag $(J_1,J_2)$.
The base of the fiber bundle is parameterized by $\varepsilon$--conjugacy classes $\{Y^{-1}\}$ in $GL_{2m+1}(F)$ such that $Y$ is the $GL_{2m+1}(F)$--component of an $F$--rational solution to (1.1) when $GL_{2m+1}\times G(m)$ is considered as a Levi subgroup of
$G(3m+1)$. We now may use $\theta^*$--stable Cartan subgroups of $GL_{2m+1}$ and their $F$--isomorphism with the members
of $\{\uT_i\}$ as in the case $n=2m+1.$ Here $\{\uT_i\}$
is a complete set of representatives for the conjugacy classes of Cartan subgroups of $\uG$. We then obtain measures
$\kappa_1(\gamma_i,\gamma'_i)|D(\gamma_i)|d\gamma_i$ on the tori
$T_i$, and since we know
${\sC}_{ss}$ (or at least $\sC_{ss}^{reg}$) is in the image of $\sN_\varepsilon$, we can integrate over $\bigcup_i T_i$
for each fiber. The integral over the complete fiber bundle is then given by means of the image correspondence $\sA$,
defined precisely as in the cases $n=2m,2m+1$, while still integrating over $\bigcup_i T_i$. Then with the choices we
have made we can write
\begin{eqnarray*}
\tilde\psi(s,Z)=&\displaystyle{\sum_\alpha\omega'(\alpha)
\int\limits_{g\in G'/G'_{\varepsilon,Y}}
\int\limits_{h\in G_Z\backslash G}
\int\limits_{h_0 \in G_Z} f'(\alpha g Y\varepsilon(g)^{-1})
f_\tau(h^{-1} Zh)}\\
&\cdot\det (g Y\varepsilon(g)^{-1})|^s\, dg\, dh\, dh_0
\int\limits_{Z(G)}\xi_{\tilde L}(z^{-2} g Y\varepsilon(g)^{-1})\xi_{L'}
(z^{-1} g X h_0 h)\\
&\cdot\det |z|^{-2s} d^\times z.
\end{eqnarray*}
Lemma 4.5 is now valid for $n>2m+1$, so we may now assume $n\geq 2m$.

For each $G$--conjugacy class of Cartan subgroups of $\uG$, choose a representative $\uT_i$, and let $T'_i$ be the set of regular elements in $T_i=\uT_i(F)$.
For $\gamma\in T'_i$ let
$$
\psi_{\sA}(s,\gamma)=\sum_{\{Y^{-1}\}\in\sA(\{\gamma\}) }\psi(s,\gamma_Y),
$$
where $\{\gamma_Y\}=\{(-1)^n(I-X' Y^{-1} X)\}$ plays the role of $Z$, depending on $Y$.

Setting
$$
R(s,Z)=(2n\log q_Fs)^{-1}\sum_\alpha\omega'(\alpha)\Phi_{\varepsilon,s}
(\alpha Y,f')\Phi(Z,f_\tau)\,\mu(\bold G_Z(F)),
$$
then,  by Lemma 4.5, $\varphi(s,Z)=\psi(s,Z)-R(s,Z)$ is an entire function.
Now set 
$$\varphi_{\sA}(s,\gamma)=\sum_{\{Y\}\in\sA(\{\gamma\})}\varphi(s,\gamma_Y).$$
Also set 
$$R_{\sA}(s,\gamma)=\sum\limits_{\{Y\}\in\sA(\{\gamma\}) }R(s,\gamma_Y).$$
Then
\begin{eqnarray*}
&&\sum_i |W(\bold T_i)|^{-1}\int\limits_{T'_i}(\psi_{\sA}(s,\gamma)|D(\gamma)|-R_{\sA}
(s,\gamma)|D(\gamma)|)d\gamma\\
&&=\sum_i |W(\bold T_i)|^{-1}\int\limits_{T'_i}\varphi_{\sA}(s,\gamma)|D(\gamma)|d\gamma.
\end{eqnarray*}
For each $i$, let $\Omega_i$ be an open compact subset of $T'_i$.
Then, using Corollary 4.6, we have
$$
\int\limits_{\Omega_i}\psi_{\sA}(s,\gamma)=h_i(s)+c_s\int\limits_{\Omega_i}
\Phi_{\varepsilon,s}(\sA(\{\gamma\}),f')\Phi(\gamma,f_\tau)|D(\gamma)|d\gamma,
$$
Where $c_s=q^{bs} L(1,2ns)$ when $\bold T_i$ is elliptic.
Here $h_i(s)$ is entire.
In the case where $\bold T_i$ is non--elliptic, then Lemma 4.5 implies the integral is entire.

Thus 
\begin{eqnarray*}
&&\Res_{s=0} \sum\limits_i |W(\bold T_i)|^{-1}\int\limits_{\Omega_i}\psi_{\sA}
(s,\gamma) |D(\gamma)|d\gamma=\\
&&c\sum_i |W(\bold T_i)|^{-1}\int\limits_{\Omega_i}\Phi_\varepsilon(\sA(\{\gamma\},f')\Phi(\gamma, f_\tau) |D(\gamma)|
d\gamma+\\&&\sum_i |W(\bold T_i)|^{-1}\Res_{s=0}\int\limits_{\Omega_i}\varphi_{\sA}(s,\gamma)
|D(\gamma)|d\gamma\\
&&=c\sum_i |W(\bold T_i)|^{-1}\int\limits_{\Omega_i}\Phi_\varepsilon(\sA(\{\gamma\},f')
\Phi(\gamma,f_\tau)|D(\gamma)|d\gamma,
\end{eqnarray*}
as $\varphi_{\sA}(s,\gamma)$ is entire.
Here $c=(2n\log q)^{-1}$.
Now letting $\Omega_i\to T'_i$, we have
\begin{eqnarray*}
&& \Res_{s=0} \sum_i|W(\bold T_i)|^{-1}\int\limits_{T'_i}\psi_{\sA}(s,\gamma)|D(\gamma)|d\gamma=\\
&&\sum_i|W(\bold T_i)|^{-1}\lim_{\Omega_i\to T'_i}\Res_{s=0}
\int\limits_{\Omega_i}\psi_{\sA}(s,\gamma)|D(\gamma)|d\gamma+\\
&&\sum_i |W(\bold T_i)|^{-1}\lim_{\Omega_i\to T_i}\Res_{s=0}
\int\limits_{T'_i\backslash\Omega_i}\psi_{\sA}(s,\gamma)|D(\gamma)|d\gamma=\\
&&=c R_G(f',f_\tau)+\sum_i |W(\bold T_i)|^{-1}\lim_{\Omega_i\to T'_i}
\Res_{s=0} \int\limits_{T'_i\backslash\Omega_i}\Phi_{\varepsilon,s}
(\sA(\{\gamma\},f')\Phi(\gamma,f_\tau)|D(\gamma)|d\gamma\\
&&+\sum_i |W(\bold T_i)|^{-1}\lim_{\Omega_i\to T'_i}\Res_{s=0}
\int\limits_{T'_i\backslash\Omega_i}\varphi_{\sA}(s,\gamma)|D(\gamma)|d\gamma.
\end{eqnarray*}
Note that, in the first sum, the limit is zero by the local boundedness of (normalized) orbital integrals.
Further, in the second sum, since $\varphi_{\sA}(s,\gamma)$ is entire, the residue is independent of the choice of $\Omega_i$, and hence we fix $\Omega_i$ and drop the limit.
Thus,
\begin{eqnarray*}
&&\Res_{s=0}\langle A(s,\tau'\otimes\tau,W_0) f(e),\ \tilde v'\otimes v\rangle\\
&&=c R_G (f',f_\tau)+\Res_{s=0} \sum_i |W(\bold
T_i)|^{-1}\int\limits_{T'_i\backslash\Omega_i} \varphi_{\sA}(s,\gamma)|D(\gamma)|d\gamma,
\end{eqnarray*}
with the second term independent of the choices of $\Omega_i$.
Note that this second term depends only on the singular part of the $T_i$.

Now suppose $n<2m$.
If $n$ is odd, we consider the embedding $G({n-1\over 2})\hookrightarrow G(m)$ by
$$
h\mapsto\begin{pmatrix} I_{m-({n-1\over 2})}\\
&h\\
&&I_{m-({n-1\over 2})}\end{pmatrix}.
$$
If $n$ is even we embed $SO_n(F)$ in $G(m)$ by
$$
h\mapsto\begin{pmatrix} I_{m-(n/2)}\\
&\varphi_n(h)\\
&&I_{m-(n/2)}\end{pmatrix},
$$
where $\varphi_n\colon SO_n\hookrightarrow SO_{n+1}$ is the map we considered in Section 3.
Let 
$$n'=\begin{cases} {n-1\over 2},&\text{if $n$ is odd}\\
\frac{n}{2},&\text{if $n$ is even} \end{cases}.$$
Then we are considering $G(n')$ as a subgroup of $G(m)$.
Suppose $Y\in GL_n(F)$ and $X\in M_{n\times 2m+1}(F)$ satisfy (1.1).
Then $X'Y^{-1}X$ has rank at most $n$, and therefore at most $n$ of the eigenvalues of $I-X'Y^{-1} X$ are different from 1.
Thus, the semisimple part of the conjugacy classes in $N_\varepsilon(\{Y^{-1}\})$ all meet $G(n')$.
Let $\sC^\vee$ be the subset of classes in $\sC$ whose semisimple part meets $G(n')$.
Then $N_\varepsilon\colon\sN\to\sC^\vee$.

\begin{lemma}If $n < 2m$ then the norm correspondence $\sN_{\varepsilon}$ has finite fibers.
\end{lemma}

\begin{proof}By Lemma 1.5(b) we know $\varepsilon(Y)^{-1} Y^{-1} X=X(I-X'Y^{-1}X)$.
We assume that $X$ is in row echelon form, and the last $n-k$ rows of $X$ are zero.
This gives a decomposition $F^n=F^k\oplus F^{n-k}$, with $F^{n-k}$ the left kernel of $X$ and $X|_{F^k}$ an injection of $F^k$ into $F^{2m+1}$.
Now the matrix of $\varepsilon(Y)^{-1}Y^{-1}$ with respect to this basis is $\begin{pmatrix} A&*\\ 0&-I\end{pmatrix}$, with $A$ determined by $I-X' Y^{-1} X$.
This shows the fibers of $N_\varepsilon$ are finite.
\end{proof}

If $n$ is odd, we first pick $X_0\in M_n(F)$ for which $Y+\tilde{\varepsilon}(Y)=X_0X'_0$.
If $n$ is even, we take $X_0\in M_{n\times n+1}(F)$ for which $Y+\tilde{\varepsilon}(Y)=X_0X'_0$.
Letting $$j=\begin{cases} {2m+1-n\over 2},&\text{$n$ odd};\\
m-\frac{n}{2},&\text{$n$ even,}\end{cases}$$ we then set
$X=\begin{pmatrix} 0_j&X_0&0_j\end{pmatrix}.$
Note that $X'=\begin{pmatrix} 0_j\\ X'_0\\ 0_j\end{pmatrix}$.
Thus, $XX'=X_0 X'_0=Y+\tilde\varepsilon(Y)$.
It is a straightforward computation that
$$
I-X' Y^{-1} X=\begin{pmatrix} I_j\\ &I-X'_0 Y^{-1} X_0\\ &&I_j\end{pmatrix},
$$
and thus, almost all conjugacy classes in $\sN$ can be parameterized by regular semisimple conjugacy classes in
$\sC^\vee{}$.

We now pick a set $\{\uT_i\}$ of representatives for the conjugacy classes of Cartan subgroups in $\uG(n')$.
Note that none of these Cartans are elliptic in $\uG(m)$.
Applying Lemma 1.4 and decomposing orbits as before, we find the contribution from each orbit is 
\begin{eqnarray*}
&&\psi(s,Z)=\sum_\alpha\omega'(\alpha)
\int\limits_{g\in G'/G'_{\varepsilon,Y}}
\int\limits_{h\in G_Z\backslash G}\int\limits_{X G_Z(F)}f'(\alpha g Y\varepsilon(g)^{-1})f_\tau(h^{-1}Zh)\cdot\\
&&\cdot|\det(g Y\varepsilon(g)^{-1})|^s dg\, dh\, d(X h_0)\int\limits_{Z(G)}
\xi_{\tilde L}(z^{-2} g Y\varepsilon(g)^{-1})\xi_{L'}(z^{-1}gX h_0 h)\\
&&\cdot |\det z|^{-2s} d^\times z,
\end{eqnarray*}
where $X$ is as above.
Note that $X G_Z(F)\cong G'_{\varepsilon,Y}$, and thus we rewrite the above formula as
\begin{eqnarray*}
&&\psi(s,Z)=\sum_\alpha\omega'(\alpha)
\int\limits_{G'/G'_{\varepsilon,Y}}
\int\limits_{G_Z\backslash G}\int\limits_{G'_{\varepsilon,Y}}f'(\alpha g Y\varepsilon(g)^{-1})f_\tau(h Zh^{-1})\\
&&\cdot|\det(g Y\varepsilon(g)^{-1})|^s\, dg\, dh\, dg_0\,
\int\limits_{Z(G')}
\xi_{\tilde L}(z^{-2} g Y\varepsilon(g)^{-1})\xi_{L'}(z^{-1}gg_0X h)\\
&&\cdot |\det z|^{-2s} d^\times z,
\end{eqnarray*}
and as in 
\cite{GS1,GS2} this expression is the same as that for $n\geq 2m$, with the roles of
$G_Z$ and $G'_{\varepsilon,Y}$, as well as those of $h_0$ and $g_0$ interchanged. Define
$\psi_{\sA}(s,\gamma)=\sum\limits_{Y\in\sA(\{\gamma\})}\psi(s,\gamma_Y)$, as before. Also let $\varphi_{\sA}(s,\gamma)$
be as before. The integration over orbits is then realized as integration over $\bigcup\limits_i T_i$, with $\{\uT_i\}$
as above. The argument then follows as in the case $n\geq 2m$, verbatim.
Note however that for any $\{\gamma\}\in\sC^\vee,\
\Phi(\gamma,f_\tau)=0$, since all $\bold T_i$ are non--elliptic and the orbital integral is in $G(m)$, not
$G(n').$ Therefore, $R_G(f',f_\tau)\equiv 0$.

We now state our main result, which we have proved.

\begin{thm}Let 
$$n'=\begin{cases} \frac n2&\text{if }n\text{ is even};\\
{n-1\over 2}&\text{ if }n\text{ is odd.}\end{cases}$$
Then let $\ell=\text{{\rm min}}(n',m)$, and let $\{\uT_i\}$ be a collection of representatives for the conjugacy
classes of Cartan subgroups of $\uG(\ell)$. For each $i$, choose an open compact subset $\Omega_i$ of the regular
elements
$T'_i$ of $T_i$. Then the intertwining operator $A(s,\tau'\otimes\tau,w_0)$ has a pole at $s=0$ if and only if
$$
c R_G (f',f_\tau)+\Res_{s=0}\sum_i|W(\bold T_i)|^{-1}
\int\limits_{T'_i\backslash\Omega_i}\varphi_{\sA}(s,\gamma)|D(\gamma)|d\gamma\not= 0,
$$
for some choice of a matrix coefficient $f_\tau$ and an $f'\in C^c_\infty
(GL_n(F))$ defining a matrix coefficient $\psi_{\tau'}$ of $\tau'$.
The constant $c=(2n\log q_F)^{-1}$.
If $n<2m$, then $R_G(f',f_\tau)\equiv 0$, and thus the residue is given by the second term alone.
\end{thm}

\begin{corollary}Suppose $\tau'\simeq\tilde\tau'$.
Fix, as in Theorem 4.9, a choice of compact open subsets $\Omega_i$ of $T'_i$.
\item{(a)} The induced representation $I(\tau'\otimes\tau)$ is irreducible if and only if for some choice of matrix
coefficients
$f_\tau,\psi_{\tau'}$
$$
c R_G (f',f_\tau)+\sum_i|W(\bold T_i)|^{-1}
\Res_{s=0}\int\limits_{T_i\backslash\Omega_i}
\varphi_{\sA} (s,\gamma) |D(\gamma)|d\gamma\not= 0,
$$
where $f'\in C_c^\infty (GL_n(F))$ defines $\psi_{\tau'}$.
\item{(b)} Assume $\tau$ is generic.
If $I(\tau'\otimes\tau)$ is irreducible then $I(s,\tau'\otimes\tau),\ s\in\Bbb R$ is reducible exactly at $s_0=\pm 1/2$
or $s_0=\pm 1$, and at only one of these pairs. In this case, the complementary series is
$I(s,\tau'\otimes\tau),$ with $-s_0<s<s_0$, and the subquotients of $I(s_0,\tau'\otimes\tau)$  for $s_0=1/2,$ or $1$ are
described in 
\cite{Sh1}, namely the Langlands quotient is non-tempered and non-generic, while the unique irreducible
subrepresentation is a generalized special discrete series.
\end{corollary}

\begin{rmk}Some poles of $A(s,\tau'\otimes\tau)$ come from poles of $L(2s,\tau',sym^2)$, which are determined
in 
\cite{Sh2}.
The residue in Corollary  4.10 represents poles of this $L$--function as well as those of
$L(s,\tau'\times\tau)$. In the next section we discuss how these appear in the theory of twisted endoscopy.
\end{rmk}

\section{Connection with twisted endoscopy, and using automorphic transfer}
We now discuss the connection of the main result with the theory of twisted endoscopy 
\cite{KS}.
While our initial reasoning is similar to that of 
\cite{GS1,GS2}, the recent results on automorphic transfer of
\cite{C-K-PS-S}, 
\cite{G-R-S}, and 
\cite{JS} allow us to be more explicit.
They also allow us to make progress towards showing the local component of the automorphic transfer is the twisted
endoscopic transfer.  If we assume that poles of $L(s,\tau'\times\tau)$ should be controlled by the regular term when
$n\geq 2m,$ then we can, in fact complete the proof of this statement in that case. To begin, we let
$\chi_\tau$ be the distribution character of
$\tau$. Then $\chi_\tau$ is represented by a locally integrable function, also denoted by $\chi_\tau$ 
\cite{HC1,HC2}.
We may then choose a matrix coefficient, $f_\tau$, with the
property that, for any regular semisimple $\gamma\in G(m)$, $\chi_\tau(\gamma)=\Phi(\gamma,f_\tau)$ 
\cite{K,C2}.

Since $(\tau')^\varepsilon=(\tau')^{w_0}=\tilde\tau'=\tau'$ (by our assumption) we can extend $\tau'$ to
$GL_n(F)\ltimes\{1,\varepsilon\}$, by fixing an equivalence $\tau'(\varepsilon)$ from $\tau'$ to $(\tau')^\varepsilon$,
whose square is the identity. The $\varepsilon$--twisted character $\chi^\varepsilon_{\tau}$, is then defined, as a
distribution by $\chi^\varepsilon_{\tau'}(f')=\text{ trace }(\tau'(f')\tau'(\varepsilon))$, for any $f'\in C_c^\infty
(GL_n(F))$. Then Clozel showed that $\chi^\varepsilon_{\tau'}$ can be represented by a locally integrable function on
the $\varepsilon$--regular set, 
\cite{C1}, and Kottwitz and Rogawski \cite{KR} discuss the
existence of $\varepsilon$--pseudo coefficients.  We assume the existence of such $\varepsilon$--pseudocoefficients.
That is, we assume there is a choice,
$f_{\tau'}$, of matrix coefficient for
$\tau'$ for which
$$
\chi_{\tau'}^\varepsilon (\gamma')=\Phi_\varepsilon(\gamma',f_{\tau'}),
$$
for all $\varepsilon$--regular elements $\gamma'\in GL_n(F)$.
Choosing $f'\in C_c^\infty (GL_n(F))$ which defines $f_{\tau'}$, we then have
$$
\Phi_\varepsilon (\sA(\{\gamma\}),f')=\sum_{\gamma'\in\sA}\Delta
(\gamma,\gamma')\chi^\varepsilon_{\tau'}(\gamma'),
$$
which we denote by $\chi^\varepsilon_{\tau'}(\sA(\{\gamma\}))$.
Now if $n=2m,$ then for such a choice of $f_\tau$ and $f_{\tau'}$ (i.e.~$f'$), the regular term becomes
$$
R_G(f',f_\tau)=\sum_i \mu(\bold T_i)|W(\bold T_i)|^{-1}\int\limits_{T_i}\chi_\tau(\gamma)
\chi^\varepsilon_{\tau'}(\gamma')\,d\gamma.
$$
Thus, $R_G$ becomes a pairing between the character $\chi_\tau$ of $\tau$ and the $\varepsilon$--twisted character $\chi^\varepsilon_{\tau'}$ of $\tau'$.
Therefore, we expect non--vanishing of $R_G(f',f_\tau)$ to indicate $\tau'$ comes from $\tau$ via twisted endoscopy.
We make the following definition.

\begin{deff}
A self dual irreducible unitary supercuspidal representation $\tau'$ of $GL_{2m}(F)$ is said to be the twisted endoscopic transfer of a
supercuspidal representation $\tau$ of $G(m)$ if $R_G(f',f_\tau)\not= 0$, for some matrix coefficient $f_\tau$ of
$\tau$ and some $f'\in C_c^\infty(G)$ defining a matrix coefficient of $\tau'$.
\end{deff}

Notice that there is a definition in 
\cite{Sh2} of when an irreducible unitary self dual supercuspidal representation $\tau'$ of
$GL_{2m}(F)$ comes from $SO_{2m+1}(F)$, and this should be interpreted as indicating that $\tau'$ is an
$\varepsilon$--twisted endoscopic transfer of $\pi$ for some $\pi$. The definition above is finer, in that it indicates
from which representation $\tau'$ should transfer.

We now define another version of transfer from $SO_{2m+1}(F)$ to $GL_{2m}(F)$.
This comes from the automorphic transfer from $SO_{2m+1}(\Bbb A_K)$ to $GL_{2m}(\Bbb A_K)$, the weak form of which is established by 
{\cite {C-K-PS-S}}, and the strong form of which is established in {\cite {JS}}.
Here $K$ is an arbitrary number field.

To be more precise, let $\tau$ be an irreducible generic supercuspidal representation of $G=\uG(F)=G(m)$.
By the techniques described in 
\cite{Sh1}, we can find a number field $K$, a finite place $v_0$ of $K$
with $K_{v_0}\simeq F$, and a cuspidal automorphic representation $\pi=\otimes_v \pi_v$ of $G(\Bbb A_K)$ such that $\pi
_{v_0}\simeq\tau$, and $\pi_v$ is unramified for $v\not= v_0$ (and $v<\infty$). By {\cite {C-K-PS-S}} and {\cite {JS}}
there is an automorphic representation $\Pi=\otimes_v\Pi_v$ of $GL_{2m}(\Bbb A_K)$ with $L(s,\Pi)=L(s,\pi)$. Moreover,
if
$\rho\colon GL_{2m}(\Bbb C)\to GL_k(\Bbb C)$ is a representation, then $L(s,\Pi,\rho)=L(s,\pi,\rho\circ i)$, where
$$
i\colon Sp_{2m}(\Bbb C)\hookrightarrow GL_{2m}(\Bbb C)
$$
is the injection of (connected components of) $L$--groups.
Let $\Pi(\tau)=\Pi_{v_0}$.
Note by {\cite {JS}}, $\Pi(\tau)$ is the $v_0$ component of the automorphic transfer $\Pi'$ of
$\pi'=\otimes_v\pi'_v$ of $G(\Bbb A_{K'})$ whenever $K'_{v_0}\simeq F$ and $\pi'_{v_0}\simeq\tau$.

\begin{deff}Let $\tau$ be an irreducible generic supercuspidal representation of $G(F)$.
We call the irreducible representation $\Pi(\tau)$ defined above the {\it local (automorphic) transfer} of $\tau$ to
$GL_{2m}(F)$.
\end{deff}

One property that the local transfer satisfies is $L(s,\tau'\times\tau)=L(s,\tau'\times\Pi(\tau))$ \cite{JS}.
Note that this Rankin--Selberg $L$--function agrees with those of Shahidi 
\cite{Sh4} and by 
\cite{HT,He} these are also Artin $L$--functions. We recall the basic consequences of this equality as the following
proposition.  We refer to \cite{JS} for another statement of these results.

\begin{prop}Suppose $\tau$ is a generic irreducible unitary supercuspidal representation of $G(m)$, and $\tau'$ an irreducible unitary supercuspidal self-dual representation of $GL_n(F)$.
\begin{description}
\item[(i)]If $n>2m$, then $L(s,\tau'\times\tau)$ is always entire.
\item[(ii)]If $n=2m$, then $L(s,\tau'\times\tau)$ has a pole at $s=0$ if and only if $\tau'\cong\Pi(\tau)$.
\item[(iii)]If $n<2m$, then $L(s,\tau'\times\tau)$ has a pole if and only if $\Pi(\tau)\subset
\tau_1\times\tau_2\times\ldots\times\tau_k$, with $\tau'\simeq\tau_i\nu$ for some for some $i$ and some unramified character $\nu.$
\end{description}
\end{prop}

\begin{proof}Let $\Pi(\tau)$ be the local automorphic transfer defined in Definition 5.2.
Since $L(s,\tau'\times\tau)=L(s,\tau'\times\Pi(\tau))$, we know from 
\cite{JPSS}, or \cite{Sh4} that
$L(s,\tau'\times\Pi(\tau))$ has a pole at $s=0$ if and only if
$\Pi(\tau)\subset\tau_1\times\ldots\times\tau_k$, with $\tau_i\nu\simeq\tilde\tau'\simeq\tau'$, for some $i$ and an unramified character $\nu.$
This immediately gives (i)--(iii).
\end{proof}

Now suppose that $\tau'$ is any irreducible unitary self dual supercuspidal of $GL_{2m}(F)$.
Then by \cite{Sh2} exactly one of the two $L$--functions, $L(s,\tau',sym^2)$ and
$L(s,\tau',\wedge^2)$ has a pole at $s=0$. Further $L(s,\tau'\times\tau)=L(s,\tau',\wedge^2)L(s,\tau',sym^2)$.

Let $\psi\colon W_F\to GL_{2m}(\Bbb C)=\,^L GL_n$ be the Langlands parameter for $\tau'$.
Then recent results of Henniart 
{\cite {He2}} show that, as expected
\begin{eqnarray}
L(s,\tau',\wedge^2)&=&L(s,\wedge^2\psi),\text{ and }\\
L(s,\tau',sym^2)&=&L(s,sym^2\psi).
\end{eqnarray}
If $L(s,\tau',sym^2)$ has a pole at $s=0$, then so does $L(2s,\tau',sym^2)$, and hence $L(s,\tau'\times\tau)$ is holomorphic at $s=0$.
Now suppose that $L(s,\tau',sym^2)$ has no pole at $s=0$.
Then $L(s,\tau',\wedge^2)$ has a pole at $s=0$, and in fact 
\cite{Sh2}, Theorem 7.6(b) shows $\tau'$ must come from $SO_{2m+1}(F),$ in the sense that this is defined there.

We have seen $L(s,\tau'\times\tau)$ is entire, unless $\tau'=\Pi(\tau)$.
In this case
\begin{eqnarray}
L(s,\tau'\times\tau)&=&L(s,\tau'\times\Pi(\tau))=L(s,\tau'\times\tau')\\
&=&L(s,\tau',\wedge^2)L(s,\tau',sym^2).\nonumber
\end{eqnarray}
Thus, the pole of $L(s,\tau'\times\tau)$ is in fact controlled by the poles of $L(s,\tau',\wedge^2)$.
By 
{\cite {JS}} and {\cite {HT,He}} we must have $\psi$ factoring through $Sp_{2n}(\Bbb C)$.  Thus, if $\tau'$ is the twisted endoscopic transfer of $\tau,$ as defined above,
then it must be the local automorphic transfer of $\tau.$  We expect these two transfers to be equivalent, but in order to assert this one would have to know that the regular term controlled the pole of the local $L$--function, as defined by \cite{Sh1}.  The singular terms are all associate with elliptic tori of smaller dimension, and we therefore expect that their non-vanishing will not be twisted endoscopic transfer from $SO_{2m}.$ We hope to prove this in later work.
We summarize this below.

\begin{thm} (See Remark 5.5) Twisted endoscopic transfer is local automorphic transfer.  More precisely, let $\tau'$ be an irreducible unitary supercuspidal representation of $GL_n(F)$ and $\tau$ an irreducible unitary generic supercuspidal representation of $G(m)$.

\begin{description}
\item[(a)]If $L(s,\tau',sym^2\rho)$ has a pole at $s=0$, then $I(\tau'\otimes\tau)$ is irreducible.
\item[(b)]Suppose $L(s,\tau',sym^2\rho)$ is holomorphic at $s=0$.  Moreover, assume $\tau'\simeq\tilde\tau'.$
\begin{description}
\item[(i)]If $n>2m$ then $I(\tau'\otimes\tau)$ is irreducible.
\item[(ii)]If $n=2m$ then $I(\tau'\otimes\tau)$ is irreducible if and only if $\tau'=\Pi(\tau)$ and this is equivalent
to both $L(s,\tau'\times\tau)$ and $L(s,\tau',\wedge^2)$ having poles at $s=0$. This is also equivalent to $\tau'$
being the local automorphic transfer of $\tau.$ If $\tau$ is the twisted endoscopic transfer of $\tau,$ i.e., if the term $R_{\bold G}$ is non-vanishing, then the two transfers are the same.
\item[(iii)]If $n<2m$, then $I(\tau'\otimes\tau)$ is irreducible if and only if
$\Pi(\tau)\subset\tau_1\times\ldots\times\tau_k$, with $\tau_i\nu\cong\tau'$ for some $i$ and some unramified character $\nu.$ In this case 
{\cite {JS}}
shows that each $\tau_i$ is parameterized by a homomorphism $\psi_i\colon W_F\to GL_{2k_i}(F)$ each of which is
symplectic. Hence each $\tau_i$ is the $\varepsilon$--twisted endoscopic transfer of some $\pi_i$ on $SO_{2k_i+1},$ in the sense of \cite {Sh2}.
\end{description}
\end{description}
\end{thm}

\begin{rmk}Note that in case (iii) the pole of $L(s,\tau'\times\tau)$ is given by non-vanishing of the singular term
$$
\sum_i|W(\bold T_i)|^{-1}\int\limits_{T'_i\backslash\Omega_i}\varphi_{\sA}(s,\gamma)|
D(\gamma)|d\gamma,
$$
and the non--vanishing of a particular term in this sum must describe which groups $GL_{2k_i}(F)$ appear in this situation.
At this point we cannot make this explicit.  We further remark that if $n=2m,$ then we expect that
$R_G(f',f_{\tau})\not\equiv 0$ if and only if $\tau'=\Pi(\tau).$  However, we have so far shown that $\tau'=\Pi(\tau)$
if and only if $R_G(f',f_\tau)+R_{\text{sing}}(f',f_\tau)\not\equiv 0.$  We expect that $R_{\text{sing}}\equiv 0.$
On the other hand, we notice that if $n<2m,$ the non-vanishing of this singular term can point to the pole of either of the two
$L$--functions, and hence it becomes clear that we have yet to understand which terms of the sum match with poles of
which $L$--functions.  However, one can expect that, in the situation of Theorem 5.4(b)(iii), the tori for which there
is non-vanishing must somehow parameterize what the factors of the supercuspidal support of $\Pi(\tau)$ are.

We also remark that the work of 
{\cite{C-K-PS-S, G-R-S, JS}} should extend to the other classical groups, and
in particular to $Sp_{2n},$ and $SO_{2n}.$ In fact, the first two among these has been extended by 
\cite{Sh5,C-K-PS-S2} and Soudry's IHP lecture \cite{So}.
The analogous analysis of the residues of the standards intertwining operators
for these cases have been resolved, at least in the case where $n$ is even
\cite{Sh3, GS1,GS2}.
When this automorphic
transfer is completely understood, then results similar to the ones of this section should be obtainable, and we expect
to address this in the near future.   Further, these results, combined together, may help resolve the issue
raised in the preceding paragraph.
\end{rmk}

\centerline{{\rm AMS Subject classification: 22E50, 11S70}}
\end{document}